\documentclass[10pt]{amsart}
\usepackage[margin=4cm]{geometry}

\usepackage{amsmath,amsthm}
\usepackage{xcolor}
\usepackage{varioref}
\usepackage{hyperref}
\hypersetup{
    colorlinks,
    linkcolor={red!50!black},
    citecolor={blue!50!black},
    urlcolor={blue!80!black}
}
\usepackage{cleveref}
\usepackage{amsfonts}
\usepackage{amssymb}
\usepackage{amsthm}
\usepackage{soul}
\usepackage{tikz-cd}
\usepackage{scalerel}

\usepackage[T5,T1]{fontenc}

\usepackage{enumitem}
\setenumerate[1]{label={(\roman*)}}

\theoremstyle{theorem}
\newtheorem{theorem}{Theorem}[subsection]
\newtheorem{proposition}[theorem]{Proposition}
\newtheorem{lemma}[theorem]{Lemma}
\newtheorem{corollary}[theorem]{Corollary}
\newtheorem*{theoremA}{Main theorem}
\newtheorem*{corollaryA}{Corollary}

\theoremstyle{definition}
\newtheorem{definition}[theorem]{Definition}
\newtheorem{remark}[theorem]{Remark}
\newtheorem{example}[theorem]{Example}

\newcommand{\B}{\mathcal{B}}
\newcommand{\lay}{\mathcal{L}}
\newcommand{\dlog}{\operatorname{dlog}}
\newcommand{\ch}[1]{{#1}^{\mathbb C}}

\newcommand{\C}{\mathbb{C}}
\newcommand{\Z}{\mathbb{Z}}
\newcommand{\Q}{\mathbb{Q}}
\newcommand{\F}{\mathbb{F}}
\newcommand{\K}{\mathbb{K}}
\newcommand{\G}{\mathbb{G}}
\renewcommand\Z{\mathbb Z}

\newcommand{\Hom}{\operatorname{Hom}}
\newcommand{\Conf}{\operatorname{Conf}}

\newcommand{\bul}{\bullet}
\newcommand{\Span}{\operatorname{Span}}
\newcommand{\rk}{\operatorname{rk}}
\newcommand{\mcg}[1]{\operatorname{PMCG}(S^2,#1)}

\newcommand{\tbraid}[1]{
{\bf Br}_{#1}
}
\newcommand{\pbraid}[1]{\operatorname{PB}_{#1}}
\newcommand{\ess}[1]{
#1^{\textrm{e}}
}

\newcommand{\wt}{\widetilde}
\newcommand{\ol}{\overline}

\begin{document}

\title{
Toric arrangements and Bloch-Kato pro-$p$ groups. 
}
\author{Emanuele Delucchi}
\address{
SUPSI-IDSIA, University of applied arts and sciences of Southern Switzerland}
\email{emanuele.delucchi@supsi.ch}
\author{Ettore Marmo}
\address{
Dipartimento di Matematica,
Università degli Studi di Milano-Bicocca}
\email{ettore.marmo01@universitadipavia.it}
\date{\today}

\setcounter{tocdepth}{1}

\begin{abstract}
We prove a purely combinatorial obstruction for the Bloch-Kato property within the class of fundamental groups of complement manifolds of toric arrangements (i.e., arrangements of hypersurfaces in the complex torus). As a stepping stone we obtain a combinatorial obstruction for the cohomology of a supersolvable arrangement to be generated in degree 1.\\
Our result allows us to prove that
\begin{itemize}
\item[-] for all prime numbers $p$, the pro-$p$ completion of the pure braid group on $k$ strands has the Bloch-Kato property if and only if $k\leq 3$;
\item[-] for all prime numbers $p$, the pro-$p$ completion of the pure mapping class group of the sphere $S^2$ with $k$ punctures has the Bloch-Kato property if and only if $k\leq 4$.
\end{itemize}
\end{abstract}

\maketitle

\tableofcontents

\section{Introduction}

The goal of this paper is to apply recent developments in the geometric and combinatorial theory of arrangements of hypersurfaces in order to obtain results in the study of Bloch-Kato pro-$p$ groups. 
{To this end we establish a connection %
between algebraic number theory and combinatorial topology that may lead to further progress in the future.}

\medskip

A major research area in algebraic number theory is devoted to understanding the structure of {absolute Galois groups} (see \S\ref{subsec:bk}). An important open problem is to classify which profinite groups can arise as absolute Galois groups.
Since many questions about absolute Galois groups can be reduced to the study of their maximal pro-$p$ quotients (the {maximal pro-$p$ Galois groups}, see \S\ref{ssec:probk}), recognizing such quotients among all abstract pro-$p$ groups is an important open question. In this context, the {\bf Bloch-Kato property} of a pro-$p$ group (\Cref{def:bk1}) is a necessary cohomological condition for the group to be a maximal pro-$p$ Galois group. %
In this paper we focus on a family of groups associated to some geometric objects, namely the pro-$p$ completions of the fundamental groups of the complement space of a toric arrangement. 

\medskip

{\bf Toric arrangements} are arrangements of hypersurfaces in a complex torus (see \S\ref{subsec:toric}), and they generalize the classical notion of (complex) hyperplane arrangement which has given rise to a rich theory over the last decades \cite{OT92}. Recent work in the field has revealed a subtle interplay between the intersection pattern of the hypersurfaces and notable topological properties of the complement of the arrangement inside the ambient torus. For instance, there is a sufficient combinatorial condition that ensures that such complements are Eilenberg-Maclane spaces \cite{BD24}. Moreover, explicit presentations for the cohomology algebra of such spaces are available \cite{CDDMP20}. This allows us to study  the Bloch-Kato property of the pro-$p$ completion of the fundamental groups of toric arrangement complements.

\medskip

Before stating precisely our main results, let us stress that our techniques apply to some notable classes of groups. For instance, as an application of our work we can prove that for all primes $p$, the pro-$p$ completion of the pure braid group on at least $4$ strands is not Bloch-Kato -- hence cannot be realized as a maximal pro-$p$ Galois group. The same holds for the pure mapping class group of the two-dimensional sphere with at least $5$ punctures.

\subsection{Bloch-Kato pro-$p$ groups and maximal pro-$p$ Galois groups}\label{subsec:bk}

Let $\F$ be a field and let $G_\F$ denote the absolute Galois group of $\F$. We point to \Cref{ssec:probk} for relevant terminology and further context.\\
The Galois group of any field extension is known to be a profinite group, see \cite{RZ10} and, by a theorem of Leptin \cite{Le55}, the converse is also true: any profinite group can be realized as the Galois group of some Galois field extension. However, not all profinite groups can be realized as the absolute Galois group of some field and determining precisely for which profinite groups this can be done is an important open problem. 
A standard approach to this representability problem is to restrict the investigation to the better behaved class of pro-$p$ groups (see \S\ref{ssec:procom}). That is, instead of studying the absolute Galois groups $G_\F$, we consider their maximal pro-$p$ quotients $G_{\F}(p)$, called maximal pro-$p$ Galois groups (see \S\ref{sub:mgpg}). A consequence of the celebrated Norm-Residue Isomorphism theorem \cite{We08, We09}) is that the $\F_p$-cohomology algebra 
$H^\bul(G_\F(p), \F_p)$ is a {quadratic algebra}. %
Since every closed subgroup of a maximal pro-$p$ Galois group is again a maximal pro-$p$ Galois group, the following definition is quite natural in this context.

\begin{definition}\label{def:bk1}
   A pro-$p$ group $G$ is \emph{Bloch-Kato} if for every closed subgroup $K \leq G$, the cohomology algebra $H^\bul(K, \F_p)$ is a quadratic algebra.
\end{definition}

Clearly every maximal pro-$p$ Galois group has the Bloch-Kato property, however it is unknown at this time whether the converse implication must hold as well, that is, whether a pro-$p$ group with the Bloch-Kato property can be necessarily realized as a maximal pro-$p$ Galois group. Obstructions to this implication were considered, for example, in \cite{QSV22} and \cite{Qua14}. With this in mind, finding new examples of Bloch-Kato pro-$p$ groups is an important problem in the field of algebraic number theory.

\subsection{Toric arrangements and their groups}\label{subsec:toric}

A toric arrangement is a finite family $\B$ of hypersurfaces given as level sets of characters of a finite-dimensional complex torus.
From a topological point of view, the main object of interest is the {complement manifold} $M(\B)$ obtained by removing the hypersurfaces in $\B$ from the ambient torus. See \Cref{sec:general} for the precise definitions and some context. Note that the cohomology ring $H^\bullet(M(\B),\mathbb Z)$  can be explicitly presented by generators and relations \cite{CDDMP20}.

We will be especially interested in the class of {\em supersolvable} toric arrangements. The theory of supersolvability  is a classical subject 
    \cite{FR85, Coh01} and has recently been extended beyond the case of hyperplanes \cite{BD24}. Such arrangements are characterized by a combinatorial property of their intersection pattern which, among other things, implies that the complement manifold is a $K(\pi, 1)$ space (i.e. the homotopy groups $\pi_i(M(\B))$ are trivial if $i > 1$). Thus, if $\B$ is supersolvable the cohomology of $G(\B)$ can be computed as the cohomology of $M(\B)$.  In view of \Cref{def:bk1}, in order to find obstructions to the pro-$p$ completion of $G(\B)$ being Bloch-Kato we will construct an associated, still supersolvable toric arrangement $\B'$ such that $M(\B')$ is a topological cover of $M(\B)$ and so that non-quadraticity of the cohomology of $M(\B')$ can be ascertained by inspecting the above-mentioned presentation.

\subsection{Main results and structure of the paper}

The main result of this paper can be stated as follows.
\begin{theoremA}[see \Cref{th:arrangement.group.non.bk}]
    For any primitive, essential, supersolvable arrangement $\B$ and for all prime numbers $p$ aside from a finite subset of exceptions (dependent on $\B$, possibly empty) the pro-$p$ completion $G(\B)_{\hat p}$ of the fundamental group of $M(\B)$ does not have the Bloch-Kato property.
\end{theoremA}

This result is in line with the expectations of the experts in the field that Bloch-Kato pro-$p$ groups should be ``rare'' and shows how some standard techniques from geometry may be used in the right setting to investigate purely algebraic questions of this kind.
As a special case, since the pure braid groups can be realized as fundamental groups of primitive, strictly supersolvable toric arrangements, we obtain the following result:
\begin{corollaryA}
[see \Cref{th:braid.bk}]
    Let $k$ be a positive integer. For all prime numbers $p$, the pro-$p$ completion of the pure braid group on $k$ strands, $(\pbraid k)_{\hat p}$, and the pro-$p$ completion of the pure mapping class group of the sphere $S^2$ with $k+1$ punctures, $(\mcg{k+1})_{\hat p}$, have the Bloch-Kato property only if $k\leq 3$.
\end{corollaryA}

Since our work touches on aspects from two rather distinct areas, 
 we will start by recalling the terminology 
 and 
 some background about both Bloch-Kato pro-$p$ groups and toric arrangements. This will be done in \Cref{sec:setup}. After that, we will proceed to prove some general results group actions on essential toric arrangements. This part of the paper holds without supersolvability assumptions and may be of independent interest. In particular,
\begin{itemize}
\item 
In \Cref{sec:covers} given a primitive toric arrangement $\B$ we characterize the primes $p$ for which the complement manifold $M(\B)$ admits a {\em primitive $p$-cover}, i.e., a $p$-sheeted topological covering by the complement of another primitive toric arrangement (\Cref{prop:existence.primitive.p.cover}). From this characterization it follows that, for any  given $\B$, primitive $p$-coverings exist for all but finitely many primes $p$. We leave it as an open question to determine whether such set of primes is combinatorially determined, see \Cref{remCp}.
\item In \Cref{sec:action} we first consider a general toric arrangement $\B$ with an action of a finite group and we derive  a condition that ensures that the induced action on $H^\bul(M(\B),\mathbb Z)$ fixes the subring generated in degree $1$ but acts nontrivially on the top-degree cohomology group (\Cref{prop:action.on.cohomology}). This is accomplished by examining the induced action on the presentation for the cohomology ring given in \cite{CDDMP20}. We then obtain a characterization of the primes $p$ for which a primitive toric arrangement $\B$ admits a primitive $p$-cover whose cohomology algebra is not generated in degree one (\Cref{th:p.cover.cohomology.not.one.generated}).
\end{itemize}
We then focus on supersolvable arrangements.
\begin{itemize}
\item  In \Cref{sec:bloch.kato} we show that for every supersolvable toric arrangement $\B$ there an isomorphism  of rings between the cohomology of the pro-$p$ completion of the arrangement's group $G(\B)$ and the cohomology of the complement $M(\B)$ with coefficients in the finite field $\F_p$ (\Cref{prop:cohom.comparison}). This is the stepping stone to \Cref{th:arrangement.group.non.bk}, where, given a supersolvable primitive toric arrangement $\B$, we characterize the primes $p$ for which the pro-$p$ completion of $G(\B)$ is not Bloch-Kato and, hence, cannot be realized as a maximal pro-$p$ Galois group. 
\item 
We end the paper with an application to pure braid groups. Indeed pure braid groups can be realized as the fundamental group of the complement of a supersolvable primitive toric arrangement, so all our results apply and allow us to prove that for all primes $p$ the pro-$p$ completion of the pure braid group on $n>3$ strands is not Bloch-Kato (\Cref{th:braid.bk}.(1)). This implies an analogous non-Bloch-Kato result for the pure mapping class group of the ($n+2$)-fold punctured sphere $S^2$ (\Cref{th:braid.bk}.(2)).
\end{itemize}

\subsection{Note} During the final stages of preparation of this manuscript we became aware of the paper \cite{Koch25} whose results lead to an independent %
proof of our Theorems 5.1.3 and 5.2.4. The techniques employed in \cite{Koch25} are significantly different than the ones developed in our paper and indeed they provide different obstructions. 
In particular, as \cite{Koch25} is concerned with coherence properties, their aim is to prove the existence of subgroups that are finitely generated but not finitely presented.
On the other hand, the finite-index subgroups of a toric arrangement group $G(\B)$ constructed in \Cref{sec:covers} are themselves toric arrangement groups and so they are always finitely presented. Moreover, the toric arrangement associated to any such finite-index subgroup can be described explicitly.

\subsection{Acknowledgements} The authors thank Thomas Weigel for introducing them to each other and encouraging the collaboration that led to this work. The second author is a member of Gruppo Nazionale per le Strutture Algebriche, Geometriche e le loro Applicazioni (GNSAGA) which is part of the Istituto Nazionale Di Alta Matematica (INDAM).

\section{Setup and notation}\label{sec:setup}

\subsection{Pro-$p$ groups and the Bloch-Kato property}\label{ssec:probk}

\subsubsection{Profinite and pro-$p$ groups}\label{ssec:procom}
Let $(I, \preceq)$ be a \emph{directed} poset, that is a partially ordered set $I$ such that for all $i, j \in I$ there exists $k \in I$ with $i, j \preceq k$. A family of finite groups $\{G_i\}_{i \in I}$ and group homomorphisms $\{\varphi_{ji} : G_j \to G_i \mid i \preceq j \}_{i, j \in I}$ indexed by a directed poset $(I, \preceq)$ is called an \emph{inverse system} of groups if whenever $i \preceq j \preceq k$ the following triangle commutes:
\[\begin{tikzcd}
    G_k \arrow[rr, "\varphi_{ki}"] \arrow[dr, "\varphi_{kj}" below left] & & G_i \\
    & G_j \arrow[ur, "\varphi_{ji}" below right] &
\end{tikzcd}\]  
    
From the datum of an inverse system of finite groups $\Phi = (\{G_i\}_{i \in I}, \{\varphi_{ij}\}_{i, j \in I})$ we can construct a new group, called the \emph{inverse limit} of $\Phi$, defined by
\[
    \varprojlim\nolimits_{i \in I} G_i := \left\lbrace (g_i)_{i\in I} \in \prod_{i \in I} G_i \mid \varphi_{ji}(g_j) = g_i \ \forall i \preceq j \right\rbrace 
\]
A \emph{profinite group} is a group that can be expressed as the inverse limit of some inverse system of finite groups. These groups are equipped with a natural topology and they can be equivalently characterized as the topological groups which are Hausdorff, compact and totally disconnected. Such groups arise naturally in algebraic number theory, since the Galois group of any Galois field extension is a profinite group \cite{RZ10}. \\
A natural way to construct profinite groups is given by the \emph{profinite completion}. Let $G$ be a discrete group and define a directed poset $\mathcal N$ by considering all the normal, finite-index subgroups $N \unlhd G$ ordered by reverse containment (i.e, $N\preceq N'$ if $N\supseteq N'$). The profinite completion $\widehat G$ of $G$ is then the inverse limit of the inverse system given by the finite quotients of $G$ and the natural projection maps
\[
    \widehat G := \varprojlim\nolimits_{N \in \mathcal N} G/N
\]
The universal property of the inverse limit yields a
canonical map $\rho_G : G \to \widehat G$ which is injective if and only if $G$ is residually finite, i.e. such that $\bigcap_{N \in \mathcal N} N = \{1\}$.
\medskip

For each prime number $p$ we can define the class of \emph{pro-$p$ groups} consisting of the inverse limits of finite groups of cardinality a power of $p$ (i.e. finite $p$-groups). Of course every pro-$p$ group is also a profinite group and many problems about profinite groups can be reduced to studying pro-$p$ groups.
Starting from a discrete group $G$ and a prime number $p$ we can define a directed poset ${\mathcal N}_p$ consisting of the normal subgroups $N \unlhd G$ such that the quotient $G/N$ is a finite $p$-group. The inverse limit of the inverse system given by such quotients is called the pro-$p$ completion of $G$:
\[
    G_{\widehat p} := \varprojlim\nolimits_{N \in {\mathcal N}_p} G/N
\]
As in the case of the pro-$p$ completion, the universal property of the inverse limit yields a canonical map $\rho_{G, p} : G \to G_{\widehat p}$  
 which is injective if and only if $G$ is ``\emph{residually $p$}'', that is if $\bigcap_{N \in \mathcal N_p} N = \{1\}$.
For a complete exposition about the theory of profinite and pro-$p$ group we refer to \cite{RZ10} 

\subsubsection{Maximal pro-$p$ Galois groups}\label{sub:mgpg}

Let $\F$ be a field and $\K \supset \F$ an algebraic extension (written $\K/\F$). An element $\alpha \in \K$ is said to be separable over $\F$ if its minimal polynomial $f(x) \in \F[x]$ has pairwise distinct roots. The extension $\K/\F$ is then called \emph{separable} if every element of $\K$ is separable over $\F$. 
Let $\F^{\rm alg}$ be an algebraic closure of $\F$. The largest subfield $\K$ of $\F^{\rm sep}$ containing $\F$ such that $\K/\F$ is a separable extension is called the \emph{separable closure} of $\F$ and is denoted by $\F^{\rm sep}$.

The extension $\F^{\rm sep}/\F$ is a Galois extension and its Galois group $G_\F$ is called the \emph{absolute Galois group} of $\F$ and it encodes in its quotients all the Galois groups of the ``well-behaved'' extensions of $\F$. As such, absolute Galois groups are one of the central objects in the algebraic theory of numbers.

For a prime number $p$, the largest quotient $G_\F(p)$ of $G_\F$ that is a pro-$p$ group is called the \emph{maximal pro-$p$ Galois group} of $\F$. Recognizing which pro-$p$ groups can be realized as maximal pro-$p$ Galois groups is one of the challenges of modern Galois theory and it is a first step to understanding the structure of absolute Galois groups.
The proof of the celebrated Norm-Residue theorem by Rost and Voevodsky states that under mild conditions ($\F$ contains a $p$-th root of unity) the $\F_p$ cohomology of $G_\F(p)$ is a quadratic algebra.

\begin{definition}\label{def:quadratic.algebra}
   A graded algebra $A^\bul$ is \emph{quadratic} if it is generated in degree one and its relation ideal is  generated in degree $2$, in other words if
   \[
       A^\bul \cong T^\bul(A^1)/(R) \quad \text{with} \quad R \subset T^2(A^1) = A^1 \otimes A^1
   \]
   where $T^\bul(V)$ denotes the free tensor algebra over a vectorspace $V$.
\end{definition}

Since every closed subgroup of a maximal pro-$p$ Galois group is again a maximal pro-$p$ Galois group, the Norm-Residue theorem implies that $G_\F(p)$ is an example of Bloch-Kato pro-$p$ group.

\subsection{Toric arrangements}
\subsubsection{Generalities}\label{sec:general}

Let $\Lambda$ be a free abelian group of rank $d$ and consider the (multiplicative) group $\mathbb C^\times$. We consider these as topological groups with respect to the discrete topology on $\Lambda$ and the usual topology of $\mathbb C^\times$. This induces a topology on the set 
$T_\Lambda:=\Hom(\Lambda,\mathbb C^\times)$ that makes it homeomorphic to the complex torus $(\mathbb C^\times)^d$. To every nonzero $\chi\in \Lambda$ corresponds a character $\ch{\chi}:T_\Lambda \to \mathbb C^\times$ defined by $\ch{\chi}(\varphi)\mapsto \varphi(\chi)$. Given $b\in \mathbb C^\times$ define
\[
H_{(\chi,b)} := \chi^{-1}(b)\subseteq T_{\Lambda}.
\]
Then $H_{(\chi,b)}$ is a hypersurface in $T_\Lambda$. We will write $H_{\chi}:=H_{(\chi,1)}$ as a shorthand for the kernel of the character $\chi$.

\begin{definition}[Toric arrangement]\label{def:toric.arrangement}
Let $\Lambda$ be a free abelian group of rank $d$. Any choice of nonzero elements $\chi_1,\ldots,\chi_n\in \Lambda$ and $b_1, \dots, b_n \in \C^\times$ defines a {\em toric arrangement} in $T_\Lambda \simeq (\mathbb C^\times)^d$ which we denote by
$$
\B:=\{H_1,\ldots,H_n\},\quad
\textrm{ where } H_i= H_{(\chi_i, b_i)}.
$$
The {\em complement} of the arrangement and its fundamental group are
$$
	M(\B) := T \setminus \bigcup_{H \in \B} H,\quad\quad
G(\B) := \pi_1(M(\B), *)
$$
(note that $M(\B)$ is always path-connected).
\end{definition}

\begin{remark}[Defining characters]
If no confusion can arise, we will identify elements of $\Lambda$ and characters of $T_\Lambda$ via the canonical isomorphism $\chi \mapsto\ch{\chi}$. Under this identification, the $\chi_i$ in \Cref{def:toric.arrangement} are called {\em defining characters} of $\B$.
\end{remark}

\begin{remark}[Central arrangements]
In the literature, a toric arrangement as in \Cref{def:toric.arrangement} is called {\em central} if $b_i=1$ for all $i$, i.e., the hypersurfaces are kernels of characters.
\end{remark}

\begin{remark}[The case $\Lambda=\mathbb Z^d$] \label{caso.Z} When $\Lambda = \mathbb Z^d$ we can use the canonical basis ${e}_1, \dots, {e}_d$ to write each element $\chi \in \Z^d$ as a linear combination $\chi = a_1^\chi {e}_1 + \dots a_d^\chi {e}_d$ or, equivalently, as a $d$-tuple $(a_1^\chi, \dots, a_d^\chi)$. With this notation we can write down explicitly the homeomorphism $T_\Lambda \to (\C^\times)^d$ as $\alpha \mapsto (\alpha({e}_1), \dots, \alpha({e}_d))$ for $\alpha : \Lambda \to \C^\times$. Under this identification, the character $\ch{\chi} : (\C^\times)^d \to \C^\times$ associated to $\chi$ is usually written multiplicatively, hence for $b\in \mathbb C^\times$ the hypersurface $H_{(\chi,b)}$ is given by 
\[
    H_{(\chi, b)} = \left\{ \textbf{t} = (t_1,t_2, \dots, t_d) \in (\C^\times)^d \mid \ch{\chi}(\textbf{t}) = t_1^{a_1^\chi}t_2^{a_2^\chi} \cdots t_d^{a_d^\chi} = b
    \right\}
\]
\end{remark}

\begin{definition}[Poset of layers] 
Let $\B$ be as in \Cref{def:toric.arrangement}. An {\em intersection} of $\B$ is any nonempty subset of $T_\Lambda$ of the form $\bigcap_{H\in \B'}H$ for some $\B'\subseteq\B$.  Connected components of intersections are called {\em layers}. 
We write $\lay(\B)$ for the {\em poset of layers}, i.e., the set of all layers partially ordered by reverse inclusion ($X\leq Y$ if $X\supseteq Y$).  See \Cref{fig_e1}.
\end{definition}

The partial order $\lay(\B)$ is commonly understood as ``the combinatorial data'' of the arrangement. Accordingly, one says that a property of $\B$ is combinatorial if it is determined by $\lay(\B)$.

\begin{remark} If $b_i\in S^1$ for all $i$ (e.g., if $\B$ is central) the poset $\lay(\B)$ is completely determined by the intersections of the induced arrangement in the compact torus $(S^1)^d$. 
\end{remark}

\begin{example}\label{ex1} 
Consider the central arrangement $\B$ defined in the two-dimensional torus $T=\Hom(\mathbb Z^2,\mathbb C^\times)\simeq (\mathbb C^\times)^2$ by the characters $\chi_1=(1,0), \chi_2=(1,-1)$ and $\chi_3=(1,1)$. The induced arrangement in the compact torus and the poset of layers $\lay(\B)$ are depicted in \Cref{fig_e1}. 
\end{example}

\begin{figure}[h]
\centering
\begin{tikzpicture}[x=10em,y=10em]
\node[anchor=center] (A) at (0,0) {};
\node[anchor=center] (B) at (1,0) {};
\node[anchor=center] (C) at (1,1) {};
\node[anchor=center] (D) at (0,1) {};
\fill[gray!10] (A.center) -- (B.center) -- (C.center) -- (D.center) -- (A.center);
\draw[very thick] (D.center) -- (A.center);
\draw[very thick] (C.center) -- (A.center);
\draw[very thick] (D.center) -- (B.center);
\node[anchor=north east] (P) at (0,0) {$P$};
\node[anchor=south] (Q) at (0.5,0.5) {$Q$};
\node[anchor=east] (H1) at (0,0.5) {$H_1$};
\node[anchor=north west] (H2) at (0.25,0.3) {$H_2$};
\node[anchor=north east] (H3) at (0.76,0.3) {$H_3$};
\end{tikzpicture}
\quad\quad\quad
\begin{tikzpicture}[x=5em,y=4.5em]
\node[anchor=center] (O) at (0,-1) {$T$};
\node[anchor=center] (H1) at (-1,0) {$H_1$};
\node[anchor=center] (H2) at (0,0) {$H_2$};
\node[anchor=center] (H3) at (1,0) {$H_3$};
\node[anchor=center] (P) at (-.5,1) {$P$};
\node[anchor=center] (Q) at (.5,1) {$Q$};
\draw (H1.north) -- (P.south) -- (H2.north) -- (Q.south) -- (H3.north);
\draw (P.south) -- (H3.north);
\draw (O.north) -- (H1.south);
\draw (O.north) -- (H2.south);
\draw (O.north) -- (H3.south);
\end{tikzpicture}
\caption{The arrangement $\B$ from \Cref{ex1} and the poset of layers $\lay(\B)$}\label{fig_e1}
\end{figure}
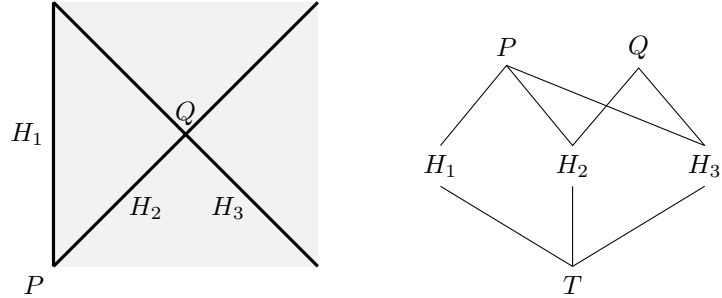

\subsubsection{Essential arrangements}
\begin{definition}[Essential arrangements]
An arrangement $\B$ is \emph{essential} if the maximal elements of the poset of layers $\lay(\B)$ have dimension $0$. Equivalently, the elements $\chi_1,\ldots,\chi_n$ generate a full-rank sublattice of $\Lambda$.
\end{definition}

\begin{lemma} \label{lem:essentialization}
Let $\B$ be a central toric arrangement in $T_\Lambda$ with set of defining characters $X:=\{\chi_1,\ldots,\chi_n\}$, and let $m:=\rk(\Lambda) - \rk(\langle \chi_1, \dots, \chi_n\rangle_\Z)$. Then there is a sublattice $\Lambda'\subseteq \Lambda$ containing $X$ and such that
\begin{enumerate}
\item the toric arrangement $\ess{\B}$ defined by $X$ in $T_{\Lambda'}$ is essential; 
\item there is an isomorphism of posets $\lay(\B)\simeq \lay(\ess{\B})$, and 
\item there is a homeomorphism
$
M(\B) \cong  M(\ess{\B})\times (\mathbb C^\times)^m
$.
\end{enumerate}
\end{lemma}
\begin{proof}
Consider $\Lambda':=\Lambda\cap\langle \chi_1,\ldots,\chi_n\rangle_\Q
$, the sublattice of $\Lambda$ consisting of all elements that are rational combinations of the $\chi_i$. Then $\rk(\Lambda') = \rk(\langle \chi_1, \dots, \chi_n\rangle_\Z)$. The short exact sequence of abelian groups
\[\begin{tikzcd}
    0 \arrow[r] & \Lambda' \arrow[r, hook] & \Lambda \arrow[r, two heads] & \Lambda/\Lambda' \arrow[r] & 0
\end{tikzcd}\]
splits, so there is a sublattice $\Lambda''\subseteq \Lambda$ of rank $m$ such that $\Lambda=\Lambda'\oplus \Lambda''$. This induces a homeomorphism
$$
\Phi: \Hom(\Lambda,\mathbb C^\times) \to 
\Hom(\Lambda',\mathbb C^\times)
\times
\Hom(\Lambda'',\mathbb C^\times)
$$
sending $\varphi\in \Hom(\Lambda,\C^\times )$ to the pair $(\varphi',\varphi'')$ defined by the restrictions $\varphi' :=\varphi|_{\Lambda'}$ and $\varphi'' :=\varphi|_{\Lambda''}$. The inverse of $\Phi$ maps an element $(\alpha, \beta) \in \Hom(\Lambda', \C^\times) \times \Hom(\Lambda'', \C^\times)$ to  $(\alpha \oplus \beta)(x) = \alpha(x')\beta(x'')$ where $x = x' \oplus x'' \in \Lambda$.

For every $i$ let $\ess{\chi_i}\in \Lambda'$ be the element such that $\chi_i=\ess{\chi_i}\oplus 0$. Now, given $\varphi \in H_{\chi_i} \subset \Hom(\Lambda, \C^\times)$, we have $\varphi' \in H_{\chi_i^e}$ therefore $\Phi(H_{\chi_i}) \subseteq H_{\chi_i^e} \times T_{\Lambda''}$. Conversely for every $(\alpha, \beta) \in H_{\chi_i^e} \times T_{\Lambda''}$, we have $\Phi^{-1}(\alpha, \beta)(\chi_i) = (\alpha\oplus\beta)(\chi_i^e\oplus 0)= \alpha(\chi_i^e)\beta(0) = 1$, so the reverse inclusion also holds. %
Thus we have proven:
\begin{equation}\label{ilpunto}
    \Phi(H_{\chi_i})= H_{\ess{\chi_i}}\times T_{\Lambda''}.
\end{equation}

Now, let $\B^e = \left\{ H_{\chi_i^e} \mid \chi_i \in X \right\}$. \Cref{ilpunto} implies immediately that $\Phi(M(\B))= M(\ess{\B})\times T_{\Lambda''}$. Claim (3) follows because $\Phi$ is a homeomorphism and since $\Lambda''$ has rank $m$.

On the other hand, \Cref{ilpunto} also implies that the function $\lay
(\ess{\B})\to\lay(\B)$ given by $L\mapsto L\times T_{\Lambda''}$ is an order-preserving bijection. Since $\lay(\B)$ is a finite poset, this implies claim (2).

Claim (1) easily follows because by construction $\Lambda'$ has the same rank as the set $\{\chi_1,\ldots,\chi_n\}$.
\end{proof}

\begin{definition}[Essentialization]
The arrangement $\ess{\B}$ defined in \Cref{lem:essentialization} is called the {\em essentialization} of $\B$. If $\B$ is essential, $\B=\ess{\B}$.
\end{definition}

\begin{remark} 
Let $\B$ be a central toric arrangement and let $\ess{\B}$ be its essentialization. By \Cref{lem:essentialization}.(3), the natural projections define an isomorphism
$$
\pi_1(M(\B)) \simeq \pi_1(M(\ess{\B}))\times \Z^m.
$$ 
\end{remark}

\begin{example}[Toric braid arrangements]
\label{ex2}
 Given an integer $n>2$, the toric braid arrangement $\tbraid{n}$ is the central toric arrangement in 
 $T_{\Z^n}=(\mathbb C^\times)^n$ with set of defining characters $\Phi_n:=\{e_i-e_j \mid 1\leq i<j\leq n\}\subset \mathbb Z^n$ (here the $e_i$ denote the standard basis of $\mathbb Z^n$). 
 The hypersurface associated to the character $e_i-e_j$ has equation $z_iz_j^{-1}=1$, and so the intersection of all elements of $\tbraid{n}$ is the set $(1,1,\ldots,1)\mathbb C^\times=\{z\in (\mathbb C^\times)^n \mid z_i=z_j\,\textrm{ for }i\neq j\}$, a connected subgroup of $(\mathbb C^\times)^n$ of rank $1$. This implies  that $\tbraid{n}$ is not essential and, moreover, that $\lay(\tbraid{n})\simeq \lay(\ess{\tbraid{n}})$ has a unique maximal element. %
 In \Cref{fig_e2} we show $\ess{\tbraid{3}}$ and its poset of layers. As the defining characters of $\tbraid{n}$ are roots of the Coxeter system $A_n$, this arrangement is also called {\em type $A_n$} toric arrangement.
\end{example}

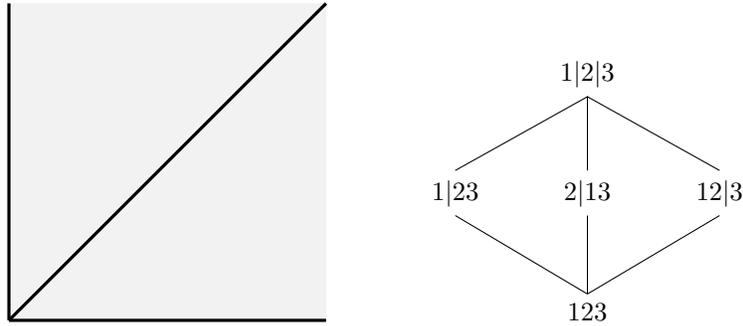
\begin{figure}[h]\label{fig2}
\centering
\begin{tikzpicture}[x=12em,y=12em]
\node[anchor=center] (A) at (0,0) {};
\node[anchor=center] (B) at (1,0) {};
\node[anchor=center] (C) at (1,1) {};
\node[anchor=center] (D) at (0,1) {};
\fill[gray!10] (A.center) -- (B.center) -- (C.center) -- (D.center) -- (A.center);
\draw[very thick] (D.center) -- (A.center);
\draw[very thick] (C.center) -- (A.center);
\draw[very thick] (A.center) -- (B.center);
\end{tikzpicture}
\quad\quad\quad
\begin{tikzpicture}[x=5em,y=4.5em]
\node[anchor=center] (O) at (0,-1) {$123$};
\node[anchor=center] (H1) at (-1,0) {$1|23$};
\node[anchor=center] (H2) at (0,0) {$2|13$};
\node[anchor=center] (H3) at (1,0) {$12|3$};
\node[anchor=center] (P) at (0,1) {$1|2|3$};
\draw (H1.north) -- (P.south) -- (H2.north);
\draw (P.south) -- (H3.north);
\draw (O.north) -- (H1.south);
\draw (O.north) -- (H2.south);
\draw (O.north) -- (H3.south);
\end{tikzpicture}
\caption{A picture of the arrangement $\ess{\tbraid{3}}$ of \Cref{ex2} and its poset of layers. One can see the correspondence with the lattice of partitions of $3$ elements: the partition with block structure $b_1|b_2|\ldots|b_k$ corresponds to the subspace given by the equations $z_i=z_j$ if there is $l$ such that $i,j\in b_l$.  }\label{fig_e2}
\end{figure}

\begin{example}[Configuration spaces] \label{ex3}
The union of the hypersurfaces in the arrangement $\tbraid{n}$ is exactly the set of all points in $(\mathbb C^\times)^n$ with distinct coordinates. Thus $M(\tbraid{n})=\Conf_n(\mathbb C^\times)$, the configuration space of $n$ points in the punctured plane $\mathbb C^\times$. Mapping a configuration of $n+1$ distinct points $\{z_1,\ldots,z_{n+1}\}\subset\mathbb C$ to the configuration $\{z_1-z_{n+1},\ldots,z_{n}-z_{n+1}\}$ defines a homeomorphism between $\Conf_{n+1}(\mathbb C)$, the configuration space of $n+1$ points in $\mathbb C$, and $\Conf_{n}(\mathbb C^\times)$. Thus the fundamental group of the arrangement $\tbraid{n}$ is the pure braid group on $n+1$ strands $\pbraid{n+1}$.
\end{example}

\subsubsection{Fiber-type and supersolvable  arrangements} 

We review some essentials of the theory of fiber-type and supersolvable toric arrangements as developed in \cite{BD24}. 

\begin{definition}[Fiber-type arrangement]
An essential toric arrangement $\B$ in $T_\Lambda$ is called {\em fiber-type} if either $\rk(\B)=1$ or there is a direct sum decomposition $\Lambda=\Lambda_0\oplus\Lambda_1$ with $\rk(\Lambda_0)=1$ such that
\begin{enumerate}
\item the arrangement $\B_1$  in $T_{\Lambda_1}$ defined by the projections of the $\chi_i$ in $\Lambda_1$ is fiber-type.
\item 
 the projection $T_{\Lambda}\to T_{\Lambda_1}$ induced by the inclusion $\Lambda_1\hookrightarrow \Lambda$ restricts to a fiber bundle $M(\B)\to M(\B_1)$.
 \end{enumerate}
\end{definition}

In our context, the interest of fiber-type arrangements stems from the following fundamental property.

\begin{proposition}{\cite[Theorem 3.4.3]{BD24}}\label{prop:polyfree.and.Kpi1}
    Let $\B$ be a fiber-type toric arrangement. Then $\B$ is  $K(\pi,1)$, and the fundamental group $\pi_1(M(\B))$ is an iterated semidirect product of free groups. 
\end{proposition}

In the following we explain how the fiber-type property for a toric arrangement can be ascertained by looking at the poset of layers, and how this approach gives rise to a special class of fiber-type arrangements that will be of interest.  Let us start with some basics about posets and we refer to \cite{St12} for a more comprehensive treatment of poset theory.

A {\em poset} $(P,\leq)$ is a set $P$ with a partial order relation $\leq$. We will only consider finite posets, i.e., $\vert P \vert <\infty$. If no misunderstanding can arise, we omit explicit mention of $\leq$ and just refer to ``the poset $P$''. An {\em order ideal} in $P$ is any $I\subseteq P$ such that $x\in I$ and $y\leq x$ implies $y\in I$. 
A {\em chain} in $P$ is any totally ordered subset $x_0 < x_2 <\ldots< x_n$ with $x_i\in P$ for all $i=0,\ldots,n$, and the length of this chain is $n$, i.e., one less than its cardinality. The length of a finite poset $P$ is the maximum of the length of any chain in $P$. We say that a poset $P$ is {\em pure} if all maximal chains have the same length. We say that $P$ is {\em bounded below} if there is a unique $\hat{0}\in P$ with $x\geq \hat{0}$ for all $x\in P$. %
If $P$ is bounded below, the {\em atoms} of $P$ are the elements $a\in P$ such that $\hat{0}\lessdot a$ (recall that $x\lessdot y$ means that $x<y$ and that $x\leq z < y$ implies $x=z$). We write $A(P)$ for the set of atoms of $P$.

Given $x,y\in P$, we write $x\vee y$ for the {\em set} of minimal upper bounds of $x$ and $y$, and $x\wedge y$ for the set of maximal lower bounds of $x$ and $y$. A poset $P$ is a {\em lattice} if, for all $x,y\in P$, $\vert x \vee y \vert = \vert x\wedge y \vert = 1$.

\begin{definition}[Supersolvable posets of layers]
\label{def:supersolvable}
    Let $\lay$ be the poset of layers of a toric arrangement. An {\em M-ideal} of $\lay$ is a pure, join-closed order ideal $\mathcal I\subseteq \lay$ such that for any two distinct $a_1,a_2\in A(\lay)\setminus A(\mathcal I)$ and every $x\in a_1\vee a_2$ There is $a_3\in A(\mathcal I)$ such that $x>a_3$. An M-ideal $\mathcal I$ is a TM-ideal if, moreover, $\vert y\vee a\vert = 1$ for every $y\in \mathcal I$ and every $a\in A(\lay) \setminus A(\mathcal I)$.

    The poset $\lay$ is {\em supersolvable} if there is a chain of M-ideals $\{\hat{0}\} \subseteq \mathcal I_1\subseteq \ldots \subseteq \mathcal I_d=\lay$ such that $\mathcal I_i$ has length $i$ for all $i$. If, moreover, all $\mathcal I_i$ are TM-ideals, then $\lay$ is called {\em strictly supersolvable}. 
    A toric arrangement is called {\em supersolvable}, resp.\ {\em strictly supersolvable} if its poset of layers is.
\end{definition}

\begin{remark}{\cite[Proposition 5.1.9]{BD24}}
\label{rem:sslattice}
    If $\lay(\B)$ is a lattice, the notions of supersolvability and strict supersolvability of  \Cref{def:supersolvable} are equivalent to each other and to the classical notion of supersolvability for finite lattices \cite{St72}.
\end{remark}

\begin{example}\label{ex2.1}
    The braid toric arrangement $\tbraid{n}$ of \Cref{ex2} is strictly supersolvable for all $n$. This follows from \Cref{rem:sslattice} since it is known \cite[Example 3.1.3]{BD24} that the poset of layers of $\tbraid{n}$ is isomorphic to the geometric lattice $\Pi_n$ of all partitions of $\{1,\ldots,n\}$ ordered by refinement, and $\Pi_n$ is supersolvable in the classical sense \cite{St72}.
\end{example}

\begin{proposition}\label{prop:strictly.supersolvable.quadratic}
    An essential toric arrangement $\B$ is fiber-type if and only if $\lay(\B)$ is supersolvable. Moreover, if $\lay(\B)$ is {\em strictly} supersolvable, the cohomology algebra $H^*(M(\B),\mathbb Z)$ is quadratic. 
\end{proposition}
\begin{proof}
The first claim is proven in \cite[Theorem 3.4.3]{BD24} for the more general case of abelian arrangements. Moreover, the fundamental group $G(\B)$ of a supersolvable arrangement $\B$ can be expressed as an iterated semidirect product of free groups of finite rank by \cite[Corollary 5.3.4]{BD24}. Recall that a semidirect produt $G = N \rtimes H$ is said to be an \emph{almost-direct product} if the natural conjugation action of $H$ on $N$ induces a trivial action on the abelianization $N_{\rm ab}$.

Now, if $\lay(\B)$ is \emph{strictly} supersolvable, then then the fundamental group $G(\B)$ can be expressed as an iterated almost-direct product of free groups of finite rank. The cohomology of such groups has been investigated by Cohen in \cite{Coh10}, and in order to prove that $H^\bul(G,\Z)$ is generated in degree $1$  we give an adaptation of the argument of \cite[Theorem 3.1]{Coh10} to the case of integer coefficients.

Let $G$ be an iterated almost-direct product of free groups of finite rank, say $G = F_{n_d} \rtimes (F_{n_{d - 1 }} \rtimes \cdots \rtimes F_{n_1}) = F_{n_d} \rtimes \overline{G}$ with $F_{n_i}$ free of rank $n_i$. The fact that $G$ is an almost-direct product implies that the conjugation action of  $\overline{G}$ on the free normal subgroup $F_{n_d}$ induces a trivial action on the abelianization $(F_{n_d})_{\rm ab}$. With these hypotheses we know from \cite[Theorem 3.1]{FR85} that $G_{\rm ab} \cong (F_{n_d})_{\rm ab} \oplus \overline{G}_{\rm ab} \cong \Z^{n_d} \oplus \overline{G}_{\rm ab}$. Since the abelianization of such a group is torsion-free, we have $H^1(G, \Z) \cong G_{\rm ab}$, so we find $H^1(G, \Z) \cong H^1(F_{n_d}, \Z) \oplus H^1(\overline{G}, \Z)$ as abelian groups. 
Now, it is clear that the inclusion map $F_{n_d} \hookrightarrow G$ induces a surjective homomorphism $H^1(G, \Z) \to H^1(F_{n_d}, \Z)$, and since $H^k(F_{n_d}, \Z) = 0$ for $k > 1$ this map induces a surjective map for the cohomology rings $H^\bul(G, \Z) \to H^\bul(F_{n_d}, \Z)$. So we can apply the version of the Leray-Hirsch theorem given in \cite[Theorem 4D.1]{Hat02} to see that as algebras over $\Z$ we have an isomorphism $H^\bul(G, \Z) \cong H^\bul(F_{n_d}) \otimes H^\bul(\overline{G}, \Z)$. One can then use the same strategy as in \cite[Theorem 3.1]{Coh10} to prove inductively that $H^\bul(G, \Z)$ is generated in degree $1$.
In particular, since for any supersolvable toric arrangement $\B$ the complement space $M(\B)$ is a $K(G(\B), 1)$-space (see \cite[Corollary 3.4.4]{BD24}), the cohomology $H^\bul(M(\B), \Z) \cong H^\bul(G(\B), \Z)$ is quadratic as well.

\end{proof}

\begin{example}\label{PBKpi1}
    From \Cref{ex2,ex2.1} we see that $\tbraid{n}$ is fiber-type, indeed supersolvable. Via \Cref{ex3}, this implies that $M(\tbraid{n})$ is a $K(\pbraid{n+1},1)$ space. Moreover \Cref{lem:essentialization}.(2) {together with \Cref{prop:strictly.supersolvable.quadratic}} show that $\ess{\B_n}$ is supersolvable as well.
\end{example}

\begin{remark}\label{rem:mcg} Let $n\geq 3$ and consider the arrangement $\tbraid{n}$. By \Cref{lem:essentialization}.(3), there is an isomorphism
$$
\phi:\quad 
\pi_1(M(\ess{\B_n}))\times \Z
\longrightarrow
\pi_1(M(\B_n))
=
\pbraid{n+1}.
$$

It is known \cite[\S 9.2]{Farb} that the center of $\pbraid{n+1}$ is infinite cyclic. Let $1 \times \Z$ be the infinite cyclic factor of $G(\ess{\B_n}) \times \Z$ and denote with $(1, c)$ its generator. %
We claim that $\phi(1 \times \Z) = Z(\pbraid{n+1})$, i.e., $1\times \Z$ equals the center of $G(\ess{\B_n}) \times \Z$. \medskip \\
In order to prove this, suppose that there exists a central element $(g, c^h)$ in $G(\ess{\B_n}) \times \Z$ with $g \in G(\ess{\B_n})$ nontrivial. Then also $(g, 1)$ will be central as we can write it as a product of central elements: $(g, 1) = (g, c^h)(1, c^{-h})$. 
In particular, since the center of $G(\ess{\B_n})\times\Z$ is infinite cyclic, both $(1, c)$ and $(g, 1)$ must be power of the same generator, so there exists integers $a, b \in \Z$ such that $(g^a, 1) = (g, 1)^a = (1, c)^b \in 1 \times \Z$. However since the intersection $(G(\ess{\B_n}) \times 1) \cap (1 \times \Z)$ is trivial, we must have $g^a = 1$, i.e. the central element $(g, 1)$ must have torsion, contradicting the fact that the center of $G(\ess{\B_n})\times \Z$ is infinite cyclic and hence is torsion-free. 
This shows $\phi(1\times \Z)=Z(\pbraid{n+1})$ and we obtain an isomorphism
$$
  \pi_1(M(\ess{\B_n})) 
  \cong \pbraid{n+1}/Z(\pbraid{n+1})
  \cong \mcg{n+2}
$$
of $\pi_1(M(\ess{\B_n}))$ with $\mcg{n+2}$, the pure mapping class group of the $(n+2)$-fold punctured two-dimensional sphere, where the last isomorphism  uses the known isomorphism $\pbraid{n+1}/Z(\pbraid{n+1}) \simeq \mcg{n+2}$, see, e.g., \cite[\S 9.3, p.~252]{Farb}.
\end{remark}

\section{Finite covers of toric arrangements}\label{sec:covers}

\subsection{Lifting arrangements}\label{sec:lifting}

As before, let $\Lambda$ be a free abelian group of rank $d$ and let $\iota: \Lambda\hookrightarrow\wt\Lambda$ be an inclusion into another free abelian group $\wt\Lambda$ of the same rank.
By fixing $\Z$-bases $e_1, \dots, e_d$ and $\tilde e_1, \dots, \tilde e_d$ for $\Lambda$ and $\wt\Lambda$ respectively, we can consider the matrix $M = (m_{i, j})_{i=1}^d \in \mathbb Z^{d\times d}$ with integer coefficients representing $\iota$:
\[
    \iota(e_j) = \sum_{i=1}^d m_{i,j}\tilde e_i
\]
for all $1 \leq j \leq d$. The quotient $\wt\Lambda/\Lambda$ is then isomorphic to $\Z^d/{\rm im}(M)$ and by putting the matrix $M$ in Smith normal form, we can write ${\rm im} (M) \cong n_1\Z \oplus \cdots \oplus n_d\Z$ for some integers $n_1, \dots, n_d$ satisfying $n_1 \cdots n_d = |\det(M)|$. So $\wt\Lambda/\Lambda$ is a finite abelian group of order $|\det(M)|$ and we have a short exact sequence of abelian groups:
\[\begin{tikzcd}
    \Lambda \arrow[r, "\iota", hook] & \wt\Lambda \arrow[r, two heads] & \wt\Lambda/\Lambda.
\end{tikzcd}\]  
Applying the contravariant functor $ A \mapsto T_A = \Hom(A, \C^\times)$ we obtain a surjective homomorphism  $f = \iota^* : T_{\wt\Lambda} \to T_{\Lambda}$ between abstract tori, with $\ker(f)=\Hom(\wt\Lambda/\Lambda, \C^\times) \cong \wt\Lambda/\Lambda$. 
By general facts (see, e.g., \cite[Example 4.4.5]{BD24}), one obtains that $f$ is indeed a covering map whose group $\G$ of deck transformations is of order $\vert \det(M)\vert$, indeed isomorphic to $\wt\Lambda/\Lambda$.

\begin{definition}\label{def:lifted.arrangement}
Let $\B$ be a centered toric arrangement in $T:=T_\Lambda$ with set of defining characters $X = \{\chi_1, \dots, \chi_n\}$. 
For $i=1,\ldots,n$ let $H_i^U:=f^{-1}(H_i)$ be the lift of $H_i$ to $U:=T_{\wt\Lambda}$.
For each $q \in f^{-1}(H_i)$ we let $H_{i, q}^U$ denote the connected component of $H_i^U$ containing $q$. 
The arrangement lifted from $\B$ is defined as:
\[
    \B^U := \{ H_{i, q}^U \mid 1 \leq i \leq n; \, q \in f^{-1}(H_i)\} = \bigcup_{H_i \in \B} \pi_0(f^{-1}(H_i)),
\]
\noindent
where $\pi_0$ denotes the set of connected components.
Moreover, let
$$
\B^U_c:=\{H_1^U,\ldots, H_n^U\}
$$
be the central arrangement whose elements are the lifts of the elements of $\B$.
\end{definition}

\begin{remark} Obviously $\lay(\B^U)=\lay(\B^U_c)$ and $M(\B^U)=M(\B^U_c)$.
\end{remark}

\begin{remark}[cf.\ Corollary 4.4.13 of {\cite{BD24}}]
\label{BUss}
The poset $\lay(\B)$ is supersolvable if and only if $\lay(\B^U)$ is supersolvable. However, if $\lay(\B)$ is strictly supersolvable then supersolvability of $\lay(\B^U)$ is not necessarily strict.
\end{remark}

\begin{remark}[Defining characters for $\B^U_c$]\label{rem:defining.character.B^U_c}
    Since $H_i^U = \ker (\ch{\chi_i}\circ f)$, the arrangement $\B^U_c$ is central with defining characters $\ch{\hat{\chi_i}}:=\ch{\chi_i}\circ f$ with $\hat{\chi}_i = \iota(\chi_i) \in \wt\Lambda$.
\end{remark}

\begin{remark}[Defining characters for $\B^U$]\label{rem:defining.character.B^U}
The arrangement $\B^U$ is not necessarily central, as the $H_i^U$ are not necessarily connected. 
More precisely, if $a_i := |\pi_0(H_i^U)|$ is the number of connected components of $H^U_i$, then $a_i$ is the index of $\langle \hat{\chi}_i \rangle_{\mathbb Z}$ in $\wt\Lambda$. Thus we can pick primitive characters $\chi_i^U\in \wt\Lambda$ given by 
\[
    \chi_i^U 
    := \frac{\hat \chi_i}{a_i}  \in \wt\Lambda
\]
so that for every $q\in H_i^U$ we have
$$
H_{i,q}^U = H_{(\chi_i^U,b_{i,q})}
$$
(see \Cref{def:lifted.arrangement} and \S\ref{sec:general}) where $b_{i,q}:=\ch{(\chi_i^U)}(q)$ may be different from $1$.
\end{remark}

\begin{example}[Lifting $\tbraid{3}$]\label{exlift} Let $U$ be the covering space of $T \cong (\C^\times)^2$ defined as above
by the injection $\mathbb Z^2 \to \mathbb Z^2$, $z\mapsto Az$ where $$A := \left(\begin{matrix}
        2 & 1 \\ 0 & -4
    \end{matrix}\right).$$

We compute the lift of $(\tbraid{3}^e)^U$. A quick computation based on \Cref{ex2} and \Cref{lem:essentialization} shows that the set of defining characters of $\tbraid{3}^e$ can be chosen as $X = \{\chi_1 = (1, 0), \chi_2 = (0, 1), \chi_3 = (1, 1)\}$, corresponding to the hypersurfaces:
    \begin{align*}
        H_1 &:= \{ (t_1, t_2) \in T \mid t_1 = 1 \}, \\
        H_2 &:= \{ (t_1, t_2) \in T \mid t_2 = 1\}, \\
        H_3 &:= \{ (t_1, t_2) \in T \mid t_1t_2 = 1 \}.
    \end{align*}

    The lifts via the covering map corresponding to the 
    matrix $A$ are then given by
    \begin{align*}
        H_1^U &:= \{ (u_1, u_2) \in U \mid u_1^2 = 1 \}, \\
        H_2^U &:= \{ (u_1, u_2) \in U \mid u_1u_2^{- 4} = 1\}, \\
        H_3^U &:= \{ (u_1, u_2) \in U \mid u_1^3u_2^{-4} = 1 \}.
    \end{align*}    
Notice that $H_1^U$ will have two distinct connected components: $H^U_{1, x}$ and $H^U_{1, y}$ where $x = (1, 1)$ and $y = (-1, 1)$. Thus we have $(\tbraid{3}^e)^U = \{H_{1,x}^U,\, H_{1, y}^U,\, H_2^U,\, H_3^U,\}$, with set of defining characters $$X^U = \{ \chi_1^U = \frac 12(2, 0),\, \chi_2^U = (1, -4),\, \chi_3^U = (3, -4) \}$$
    (recall that by \Cref{rem:defining.character.B^U} we have $\chi_i^U= \frac{1}{a_i}\hat\chi_i = \frac{1}{a_i}A\chi_i$ where $a_i$ is the number of connected components of $H_i^U$).
    \Cref{fig3} depicts the covering of $\tbraid{3}^e$ by $(\tbraid{3}^e)^U$ and the poset of layers of $(\tbraid{3}^e)^U$. From it we see that while $(\tbraid{3}^e)^U$ is supersolvable like $\tbraid{3}^e$, the former is no longer strictly supersolvable. Indeed since $M$-ideals must be join-closed, any $M$ ideal must miss at least one among $H^U_2$ and $H^U_3$. But the intersection of either of those with any other hypersurface is disconnected.

\begin{figure}[h] %
\centering
\begin{tikzpicture}[x=15em,y=15em]

    \node[anchor=center] (A2) at (0, 1) {};
    \node[anchor=center] (B2) at (0.5, 1) {};
    \node[anchor=center] (C2) at (0.5, 1.5) {};
    \node[anchor=center] (D2) at (0, 1.5) {};

\fill[gray!10] (A2.center) -- (B2.center) -- (C2.center) -- (D2.center) -- (A2.center);
\begin{scope}
    \draw[very thick] (A2.center) -- (D2.center);
    \draw[very thick] (A2.center)++(0.25,0) -- (0.25, 1.5);
    \clip (A2.center) -- (B2.center) -- (C2.center) -- (D2.center) -- (A2.center);
    \draw[thick] (A2.center) --+ (2, 0.5);
    \draw[thick] (A2.center)++(0,0.25) --+ (2, 0.5);
    \draw[thick] (A2.center)++(0,0.125) --+ (2, 0.5);
    \draw[thick] (A2.center)++(0,0.375) --+ (2, 0.5);    
    
    \draw[] (A2.center) --+ (2, 1.5);
    \draw[] (A2.center)++(-0.5, 0) --+ (2, 1.5);    
    \draw[] (A2.center)++(-0.333, 0) --+ (2, 1.5);    
    \draw[] (A2.center)++(-0.166, 0) --+ (2, 1.5);
    \draw[] (A2.center)++(0.333, 0) --+ (2, 1.5);
    \draw[] (A2.center)++(0.166, 0) --+ (2, 1.5);
\end{scope}

    \node[anchor=center] (A1) at (0, 0) {};
    \node[anchor=center] (B1) at (0.5, 0) {};
    \node[anchor=center] (C1) at (0.5, 0.5) {};
    \node[anchor=center] (D1) at (0, 0.5) {};
    
\fill[gray!10] (A1.center) -- (B1.center) -- (C1.center) -- (D1.center) -- (A1.center);

    \draw[very thick] (D1.center) -- (A1.center);
    \draw[] (C1.center) -- (A1.center);
    \draw[thick] (A1.center) -- (B1.center);
\node[rotate=-90, scale=2] (arrow) at (0.25, 0.75) {$\twoheadrightarrow$};
\end{tikzpicture}
\quad
\begin{tikzpicture}[x=5em,y=4.5em]
\node[anchor=center] (O) at (0,-1) {$U$};
\node[anchor=center] (H2) at (-1.5,0) {$H_2^U$};
\node[anchor=center] (H1x) at (-0.5,0) {$H_{1, x}^U$};
\node[anchor=center] (H1y) at (0.5,0) {$H_{1, y}^U$};
\node[anchor=center] (H3) at (1.5,0) {$H_3^U$};
\node[anchor=center] (P) at (-2,1.25) {$P$};
\node[anchor=center] (Q) at (-1.5,1.25) {$Q$};
\node[anchor=center] (R) at (-1,1.25) {$R$};
\node[anchor=center] (S) at (-.5,1.25) {$S$};
\node[anchor=center] (P') at (.5,1.25) {$P'$};
\node[anchor=center] (Q') at (1,1.25) {$Q'$};
\node[anchor=center] (R') at (1.5,1.25) {$R'$};
\node[anchor=center] (S') at (2,1.25) {$S'$};

\draw (H2.north) -- (P.south);
\draw (H2.north) -- (Q.south);
\draw (H2.north) -- (R.south);
\draw (H2.north) -- (S.south);
\draw (H2.north) -- (P'.south);
\draw (H2.north) -- (Q'.south);
\draw (H2.north) -- (R'.south);
\draw (H2.north) -- (S'.south);

\draw (H3.north) -- (P.south);
\draw (H3.north) -- (Q.south);
\draw (H3.north) -- (R.south);
\draw (H3.north) -- (S.south);
\draw (H3.north) -- (P'.south);
\draw (H3.north) -- (Q'.south);
\draw (H3.north) -- (R'.south);
\draw (H3.north) -- (S'.south);

\draw (H1x.north) -- (P.south);
\draw (H1x.north) -- (Q.south);
\draw (H1x.north) -- (R.south);
\draw (H1x.north) -- (S.south);

\draw (H1y.north) -- (P'.south);
\draw (H1y.north) -- (Q'.south);
\draw (H1y.north) -- (R'.south);
\draw (H1y.north) -- (S'.south);

\draw (O.north) -- (H1x.south);
\draw (O.north) -- (H1y.south);
\draw (O.north) -- (H2.south);
\draw (O.north) -- (H3.south);

\node (bottom) at (0,-2) {};
\end{tikzpicture}

\caption{The lift $(\tbraid{3}^e)^U$ of $\tbraid{3}^e$ from \Cref{exlift} with the poset of layers $\lay((\tbraid{3}^e)^U)$. For later reference we note that the M\"obius function of the poset has value $1$ at the bottom element and $-1$, resp. $2$ at each element of rank $1$, resp.~$2$. Thus the characteristic polynomial of the poset is $q(t)=t^2 - 4t + 16$ and so the Poincaré polynomial of the complement $M((\tbraid{3}^e)^U)$ is 
$t^2q(-\frac{t+1}{t})=21t^2 + 6t + 1$ (see, e.g., \cite[Corollary 5.12]{Moc12})  
}\label{fig3}
\end{figure}

\end{example}

\subsection{Primitive arrangements}

\begin{definition}[Primitive elements and primitive arrangements]\label{def:primitive.vector}
    An element $v$ in a free abelian group of finite rank $\Lambda$ is \emph{primitive} if it is not a nontrivial multiple of another element, that is if $v$ cannot be written in the form $c w$ for some integer $c$ with $\vert c\vert > 1$ and another element $w \in \Lambda$. 
    \noindent
    \label{def:primitive.arrangement} An arrangement $\B$ is \emph{primitive} if each defining character $\chi$ is primitive in the character group $\Lambda$.
\end{definition}

\begin{remark}\label{rem:connected.H}
A hypertorus $H = H_{(\chi, b)}$ of $T_\Lambda$ is connected if and only if the corresponding character $\chi$ is a primitive element in $\Lambda$.
\end{remark}

\noindent
In the standard lattice $\Z^d$, a vector $v = (v_1, \dots, v_d)$ is primitive if and only if $\gcd(v_1, \dots, v_d) = 1$. Note that the property of being a primitive arrangement is not always preserved under lifting by covering maps. 

\begin{remark} In the context of \Cref{def:lifted.arrangement}, if $\B$ is a primitive toric arrangement then $\B^U$ is primitive, but $\B^U_c$ is not necessarily primitive (see \Cref{exlift}).
\end{remark}

\begin{definition}
	Let $\B$ be a primitive toric arrangement in a torus $T=T_\Lambda$ 
    and let $p$ be a prime number. We will say that $\B$ admits a \emph{primitive $p$-cover} if there exists a finite covering map $f : U \to T$ of degree a power of $p$ such that the lifted arrangement $\B^U_{c}$ is primitive. We will say that the cover $f:U\to T$ is {\em primitive respect to $\B$}.

\end{definition}

\begin{remark} A cover $f:U\to T$ is primitive with respect to $\B$ if and only if $\B^U=\B^U_c$.
\end{remark}

\noindent
The following is a characterization of the existence of primitive $p$-covers.

\begin{theorem}\label{prop:existence.primitive.p.cover}
	Let $\B = \{H_1, \dots, H_n\}$ be a primitive toric arrangement in $T = T_{\mathbb Z^d}$ 
    and let $X = \{\chi_1, \dots, \chi_n\} \subset \Z^d$ be the defining characters, written in coordinates as:
	\[
		\chi_i = (a_1^{\chi_i}, \dots, a_d^{\chi_i})\in \mathbb Z^d.
	\]
\noindent
Fix a prime number $p$. For any $\chi \in \Z^d$ we will denote by $\overline \chi \in \F_p^d$ the componentwise reduction modulo $p$.
Consider the map $$\phi_p : X \to \mathbb P^1(\F_p^d)$$ that assigns to each character $\chi_i$ the $\F_p$-span of $\overline{\chi}$, i.e., the point $[\overline{\chi_i}] = 
\{ \lambda \overline \chi_i \mid \lambda \in \F_p^\times \} 
\in \mathbb{P}^1(\F_p^d)$.
\noindent
Then $\B$ admits a primitive $p$-cover if and only if $\phi_p$ is not surjective. Moreover such cover can be chosen to have degree precisely $p$.
\end{theorem}

\begin{proof}

Suppose that $\phi_p$ is not surjective. Then there is a dimension $1$ subspace $L = \Span_{\F_p}(v) \leq \F_p^d$ whose associated point $[\overline{v}]\in\mathbb{P}^1(\F_p^d)$ is not in the set $\operatorname{im}(\phi_p)=\{[\overline{\chi_1}],\ldots [\overline{\chi_n}]\}$.
Up to a change of coordinates, we can assume that $v = (1, 0, \dots, 0)$ and consider the covering map $f: T_{\tilde{\Lambda}}\to T_{\mathbb Z^d}$ where the inclusion $\tilde{\Lambda}\hookrightarrow \mathbb Z^d$ is represented by the following diagonal matrix
\[
	M:=\left(\begin{matrix} 
		p & 0 & 0 & \cdots & 0 \\
		0 & 1 & 0 & \cdots & 0 \\
		0 & 0 & 1 & \cdots & 0 \\
		\vdots & \vdots & \vdots & \ddots & \vdots \\
		0 & 0 & 0 & \cdots & 1
	\end{matrix}\right)
\]
Recall that as each defining character $\chi$ is primitive, we have $\gcd(a_1^\chi, \dots, a_d^\chi) = 1$. Let $a_1^\chi = p^k b$ with $p \nmid b$ (note that such $k$ is unique, and it may a priori be $0$). Then 
\begin{align*}
	\gcd(f(\chi)) &= \gcd(pa_1^\chi, a_2^\chi, \dots, a_d^\chi) \\
		&= \gcd(p^{k+1}b, a_2^\chi, \dots, a_d^\chi) \\
		&= \gcd(p^{k+1}, a_2^\chi, \dots, a_d^\chi) \cdot \gcd(b, a_2^\chi, \dots, a_d^\chi)
\end{align*}
The second factor is equal to $1$ because it divides $\gcd(a_1^\chi, \dots, a_d^\chi) = 1$ since $\chi$ is primitive.
The first factor must equal $1$ as well, since the fact that $[\overline{v}]$ is not among the $[\overline{\chi_i}]$
means that %
there is at least one $j\geq 2$ such that coordinate $a_j^\chi$  is nonzero modulo $p$. In other words, there is at least one $j\geq 2$ such that $p \nmid a_i^\chi$. 
Now, since $\gcd(f(\chi))=1$ for all defining characters $\chi$ of $\B$, the lifted arrangement $\B_c^U$ is primitive. Moreover, a glance at the determinant of $M$ shows that the covering $f$ has degree $p$. \medskip\\
Conversely, suppose that $\phi_p$ is surjective. By way of contradiction, suppose that $\B$ admits a primitive $p$-cover $f:T_{\tilde{\Lambda}}\to T_{\mathbb Z^d}$ and let $M$ represent the inclusion $\tilde{\Lambda}\to \mathbb Z^d$. Then $\det(M)$ must be a power of $p$. If $\overline{M}$ is the componentwise reduction modulo $p$ of $M$, then $$\det(\overline{M}) \equiv \det(M) \equiv 0 \mod p,$$    
hence $\overline{M}$ is singular and we can consider an element $v\in \ker (\overline{M})$.
Surjectivity of $\phi_p$ implies that there is a defining character $\chi$ such that $v = \lambda \ol\chi$ for some $\lambda \in \F_p^\times$. Then
\begin{equation}\label{modp}
	\ol M \lambda \ol\chi \equiv \ol M v \equiv 0 \mod p.
\end{equation}
Consider $\hat\chi = M\chi$, which is a defining character of the lift $\B^{T_{\tilde{\Lambda}}}_c$ (see \Cref{rem:defining.character.B^U_c} and \Cref{def:lifted.arrangement}). Now \Cref{modp} implies $\ol M\ol\chi \equiv 0 \mod p$, and thus $p$ divides every coefficient of $\hat\chi$. But this means that $\hat\chi$ is not primitive in $\mathbb Z^d$ and thus $\B^{T_{\tilde{\Lambda}}}_c$ is not primitive, contradicting the fact that $f$ is a primitive cover.
\end{proof}

\begin{remark}\label{remCp} When $\cap_i H_i =\{1\}$, then by work of Pagaria \cite{Pa19} the defining characters can be recovered from the poset $\lay(\B)$ up to a unimodular transformation of $\Lambda$. Thus  within this class the set of $p$'s for which there are primitive $p$-covers is combinatorial. In general, for a fixed poset $\lay(\B)$ the family of all sets of defining characters of toric arrangements $\B'$ such that $\lay(\B')$ is isomorphic to $\lay(\B)$ can be effectively computed via work of Pagaria and Paolini \cite{PP21} (in particular, using the software \texttt{arithmat} \cite{PP19}). We leave it as an open question to determine whether the set of primes for which there is a primitive $p$-cover is constant across all arrangements $\B'$ with a given fixed poset of layers.
\end{remark}

\begin{corollary}\label{cor:existence.primitive.p.covers}
	Let $\B$ be a primitive toric arrangement. %
    Then $\B$ admits a primitive $p$-cover for all but finitely many primes $p$. 
\end{corollary}
\begin{proof}	
	Let $d$ denote the dimension of the ambient torus and let $p$ be a prime number. The finite vector space $\F_p^d$ admits exactly
	\[
		n(p, d) := \frac{p^d - 1}{p - 1} = 1 + p + p^2 + \cdots + p^{d-1}
	\]
	distinct lines, hence $n(p,d)$ is the number of elements of $\mathbb P^1(\F_p^d)$. 	\Cref{prop:existence.primitive.p.cover} asserts that $\B$ admits a primitive $p$-cover if there exists at least one element $L \in \mathbb P^1( \F_p^d)$ distinct from each of the lines generated by the defining characters modulo $p$. If the number of defining characters, i.e., the cardinality $|\B|$ is smaller than $n(p, d)$, then there will be such $L$. The claim follows because for fixed $d$ the quantity $n(p,d)$ is an increasing function of $p$.
\end{proof}

\begin{remark} \Cref{prop:existence.primitive.p.cover} can be used in order to show that for every fixed prime number $p$ there exists a primitive toric arrangement which does not admit any primitive $p$-cover. Such instances are constructed in the next example. 
\end{remark}

\begin{example}

    Consider the essential toric arrangement $\B_3^e$ from \Cref{ex2}. This is an essential primitive arrangement in $(\C^\times)^2$ %
    such that $|\B_3^e| = 3$. For every odd prime $p$ we have

    \begin{align*}
        n(p, 2) = \frac{p^2-1}{p-1} \geq \frac{3^2-1}{3-1} = 4 > |\B_3^e| = 3
    \end{align*}
    So, for any such prime $p$ the map $\phi_p$ cannot be surjective and we can find a primitive $p$-cover of $\B_3^e$. 
    On the other hand,   primitive $2$-covers cannot exist since $n(2, 2) = 3$.
\end{example}

\section{Group actions on cohomology rings}\label{sec:action}

\subsection{Setup}

In this section we consider a toric arrangement $\B=\{H_1,\ldots,H_n\}$ in a torus $U=T_\Lambda$ with defining characters $\chi_1,\ldots,\chi_n$. 
If a finite subgroup $\G < U$ acts on $\B$ by left multiplication (and hence without fixed points), then $\G$ acts on the ring $H^\bul (M(\B),\mathbb Z)$. We aim at describing this action in terms of the set of generators of $H^\bul (M(\B),\mathbb Z)$ described in \cite{CDDMP20}, namely the cohomology classes 
\[
	e_{W, A; B} := [\ol\eta_{W, A; B}]
\]
where the $\ol\eta_{W,A;B}$ are logarithmic differential forms described in \Cref{generatingforms} and indexed over triples $W,A,B$ where $W$ is a layer of $\B$ and $(A,B)$ is a \emph{$W$-adapted pair} in the following sense. Note that the class $e_{W,A;B}$  has degree $|A\sqcup B|$.

\begin{definition}[Independent sets and adapted pairs]\label{def:independent} A subset $I\subseteq \{1,\ldots, n\}$  will be called \emph{independent} if the defining characters indexed in $I$ are an independent subset of $\Lambda$, i.e., $\rk_\Lambda(\{\chi_i\}_{I})=\vert I \vert$.
Given a layer $W$ of a toric arrangement $\B$, a pair $(A,B)$ of subsets $A,B\subseteq [n]$ is called {\em $W$-adapted} if  $W$ is a connected component of $\cap_{i \in A} H_i$
and $B$ is such that $A\cap B=\emptyset$ and $A \sqcup B$ is independent subset.
\end{definition}

\begin{remark}\label{motivation-setup}
The notation and setup of this section is tailored to our specific goal, i.e., studying the situation developed in \Cref{sec:covers} where a given toric arrangement in a torus $T$ is lifted to a covering $U$ of $T$ (\Cref{def:lifted.arrangement}), giving rise to an arrangement in $U$  which, in this section, we call $\B$. In this case case the group of deck transformations of the covering $U\to T$, which we identify with a subgroup $\G$ of $U$, acts on the arrangement $\B$. If the covering is primitive, $\B$ has the same number of hypersurfaces as the arrangement in $T$, and $\G$ acts preserving the hyperplanes, i.e., $\sigma H = H$ for all $H\in \B$.
\end{remark}

The goal of this section is to describe the induced action of $\G$ on the cohomology algebra $H^\bul(M(\B), \Z)$ in terms of the presentation given in \cite{CDDMP20}.

\subsection{Logarithmic forms generating the cohomology} We start by recalling the construction of the differential forms $\ol\eta_{W, A; B}$ from \cite[Section 2.5]{CDDMP20}. For every $1 \leq i \leq n$ consider the logarithmic forms:
\begin{align}
	\psi_i &:= \frac{1}{2\pi\sqrt{-1}} \dlog(e^{\chi_i}) \\
	\omega_i &:= \frac{1}{2\pi\sqrt{-1}} \dlog(1 - e^{\chi_i})
\end{align}
where the $\psi_i$ are defined on the whole of the torus $U$, and the $\omega_i$ only on the arrangement's complement. 
Moreover, put
\begin{align}
\ol \omega_i := 2\omega_i - \psi_i.
\end{align}

\noindent
We will need some additional differential forms $\ol\omega_{W, A}$, whose definition depends on some auxiliary finite covers of $U$. Let a (finitely sheeted) covering map $f : \wt U \to U$ be given and recall the characters $\chi_i^{\wt U}=\frac{\hat{\chi_i}}{a_i}$ from \Cref{rem:defining.character.B^U}. Then, for all $1 \leq i \leq n$ and for all $q \in f^{-1}(H_i)$ we define the following forms on the complement of the lift of the arrangement $\B$ through $f$:
\begin{align}
	\psi_i^{\wt U} &:= \frac{1}{a_i} f^*(\psi_i) = \frac{1}{2\pi \sqrt{-1}}\dlog(e^{\chi_i^{\wt U}}). \\
	\omega_{i,q}^{\wt U} &:= \frac{1}{2\pi\sqrt{-1}} \dlog\left(1 - e^{\chi_i^{\wt U} - \chi_i^{\wt U}(q)}\right) \\
	\ol\omega_{i,q}^{\wt U} &:= 2\omega_{i,q}^{\wt U} - \psi_i^{\wt U}
\end{align}
For all ordered tuples $A=(i_1, \dots, i_k)$ of distinct indices $i_j\in [n]$,
we set
\begin{equation}\label{eq:indices}
\begin{array}{ll}
\psi_A := \psi_{i_1} \wedge \cdots \wedge \psi_{i_k},\quad &
\psi_A^{\wt U} := \psi_{i_1}^{\wt U} \wedge \cdots \wedge \psi_{i_k}^{\wt U};
\\[.7em]
	\omega_A := \omega_{i_1} \wedge \cdots \wedge \omega_{i_k},\quad&
    \omega_A^{\wt U} := \omega_{i_1}^{\wt U} \wedge \cdots \wedge \omega_{i_k}^{\wt U};
    \\[.7em]
	\ol\omega_A := \ol\omega_{i_1} \wedge \cdots \wedge \ol\omega_{i_k},\quad &
    \ol\omega_A^{\wt U} := \ol\omega_{i_1}^{\wt U} \wedge \cdots \wedge \ol\omega_{i_k}^{\wt U}.
\end{array}
\end{equation}

Moreover, recalling the notation in \Cref{def:lifted.arrangement}, given an ordered tuple $A$ of indices as above we set
\begin{equation}\label{not:HU}
H_A:=\bigcap_{i\in A} H_i,\quad\quad\quad
H_{A,q}^{\wt U} := \bigcap_{i\in A}H_{i,q}^{\wt U}\quad
\textrm{ for }q\in f^{-1}(H_i)
\end{equation}
\begin{definition}\label{def:uniqueomega}
Fix an ordered tuple of indices $A$, let $W$ be a connected component of $H_A$ and choose $p \in W$. Since the pullback map $f^*$ is injective, following \cite[Definition 4.4]{CDDMP20}  define the differential forms $\omega_{W, A}^f$ and $\ol\omega_{W, A}^f$ as the unique forms on $M(\B)$ such that:
\begin{align*}
	f^*(\omega_{W, A}^f) = \frac{1}{|H_{A, q_0}^{\wt U} \cap f^{-1}(p)|} \sum_{q \in f^{-1}(p)} \omega_{A, q}^{\wt U}
\end{align*}
and
\begin{align*}
	f^*(\ol\omega_{W, A}^f) = \frac{1}{|H_{A, q_0}^{\wt U} \cap f^{-1}(p)|} \sum_{q \in f^{-1}(p)} \ol\omega_{A, q}^{\wt U}
\end{align*}
\end{definition}

\begin{definition}[Separating cover, see {\cite[Definition 5.1]{CDDMP20}}]\label{def:separating.cover}
	Let $A$ be an ordered tuple of indices. We say that a covering map $f : \wt U \to U$ \emph{separates} $A$ if for every layer $W \in \pi_0(H_A)$ and all $i \in A$, there exists $q_i \in f^{-1}(H_i)$ such that
\[
		f(H_{A, q_i}^{\wt U}) = W.
\]
\end{definition}

\begin{remark}\label{rem:separating.covers}

In \cite[Proposition 5.3]{CDDMP20} it is proven that separating covers exist for all choices of $A$. Moreover, by \cite[Theorem 5.4]{CDDMP20} if $f$ and $g$ are both separating for $A$ then 
\begin{equation*}
	\omega_{W, A}^f = \omega_{W, A}^g \quad \textrm{ and } \quad
	\ol\omega_{W, A}^f = \ol\omega_{W, A}^g. 
\end{equation*}

\noindent
We will thus suppress the superscript and write simply
$\omega_{W, A}$, resp.~$\ol\omega_{W, A}$.
\end{remark}

\begin{proposition}[{\cite[Theorem 6.14]{CDDMP20}}] \label{prop:generators}
Let $\B$ be an essential toric arrangement. The cohomology algebra $H^\bullet(M(\B), \Z)$ is generated by the cohomology classes of the differential forms
\begin{equation}\label{generatingforms}
    \ol\eta_{W, A; B} = (-1)^{\ell(A, B)} \ol\omega_{W, A} \wedge\psi_B
\end{equation}
for all 
$W \in \pi_0(H_A)$ and all $W$-adapted pairs $(A,B)$, with $\ell(A, B)$ equal to the sign of the permutation taking the set $A \sqcup B$ (with increasing order) to the sequence obtained by concatenating $A$ and $B$ (each ordered in increasing order).
\end{proposition}

\subsection{Action on differential forms}
Let $\G$ be a finite subgroup of $U$ acting on an essential %
toric arrangement $\B$ by (left) multiplication.
In order to describe the action of $\G$ on the differential forms of \Cref{generatingforms} we will consider, for every finite cover of topological groups $f : \wt U \to U$ that is separating (in the sense of \Cref{def:separating.cover}), the subgroup $\wt \G$ of $\wt U$ obtained by lifting $\G$ (i.e., $\wt \G = f^{-1}(\G)$). 
Notice that if $\sigma \in \G$, $\tilde y \in \wt U$ and $\tilde \sigma \in f^{-1}(\sigma)$, we will have
\begin{equation}\label{eq:cover.homomorphism}
	f(\tilde \sigma \tilde y) = \sigma f(\tilde y).
\end{equation}
As $\G$ acts without fixed points, so does $\wt \G$.

\begin{lemma}\label{lem:psi.action}
	For $1 \leq i \leq n$ and finite cover $f : \wt U \to U$, the following holds:
	\begin{enumerate}
		\item $\psi_i$ is invariant under the action of $\G$;
		\item $\psi_i^{\wt U}$ is invariant under the action of $\wt \G$.
	\end{enumerate}
\end{lemma}
\begin{proof}
Let $\sigma \in \G$ and $1 \leq i \leq n$. Then, since $\B$ is central we have $H_i=\ker(\chi_i)$, hence $1\in  H_i$. Since $\G$ preserves $\B$ we must have $\chi_i(\sigma)=\chi_i(\sigma\cdot 1) = \chi_i(1)=1$. Then, for all $y \in U$ we have $\chi_i(\sigma y)=\chi_i(\sigma)\chi_i(y)=\chi_i(y)$ and so 
\begin{align*}
	(\sigma^* \psi_i)_y &= \frac{1}{2\pi\sqrt{-1}} \dlog(e^{\chi_i(\sigma  y)})= \frac{1}{2\pi\sqrt{-1}} \dlog(e^{\chi_i(y)})
\end{align*}
in other words $\sigma^* \psi_i = \psi_i$. This proves directly claim (i), which implies claim (ii) by equivariance of $f$: indeed, for any $\tilde\sigma \in f^{-1}(\sigma)$ we have
\begin{align*}
	\tilde \sigma^* \psi_i^{\wt U} &= \frac{1}{a_i} \tilde \sigma^* f^*(\psi_i) = \frac{1}{a_i} (f \tilde \sigma)^*(\psi_i) \\
		&= \frac{1}{a_i} (\sigma f)^*(\psi_i) = \frac{1}{a_i} f^*(\sigma^* \psi_i) = \frac{1}{a_i}f^*(\psi_i) = \psi_i^{\wt U}.
\end{align*}
\end{proof}

\begin{lemma}\label{lem:omega.action}
	For every $1 \leq i \leq n$, finite cover $f : \wt U \to U$ and $q \in f^{-1}(H_i)$ we have
	\begin{enumerate}
		\item $\tau^* \omega_{i, q}^{\wt U} = \omega_{i, \tau^{-1}  q}$; and
		\item $\tau^* \ol\omega_{i, q}^{\wt U} = \ol\omega_{i, \tau^{-1}  q}$;
	\end{enumerate}
	for every $\tau \in \wt \G$.
\end{lemma}
\begin{proof}
	Let $y \in \wt U$, we have:
	\begin{align*}
		(\tau^* \omega_{i, q}^{\wt U})_y &= \frac{1}{2\pi\sqrt{-1}} \dlog\left(1 - e^{\chi_i^{\wt U}(\tau  y) - \chi_i^{\wt U}(q)}\right) = \frac{1}{2\pi\sqrt{-1}} \dlog\left(1 - e^{\chi_i^{\wt U}(\tau y q^{-1})}\right) \\
		&=  \frac{1}{2\pi\sqrt{-1}} \dlog\left(1 - e^{\chi_i^{\wt U}(y) - \chi_i^{\wt U}(\tau^{-1}  q)}\right) = (\omega_{i, \tau^{-1} \cdot q})_y
	\end{align*}
	The second equation is clear from the $\wt \G$ invariance of $\psi_i^{\wt U}$ (\Cref{lem:psi.action}.(ii)) and the first equation.
\end{proof}

\begin{corollary}\label{cor:psi.omega.A.action}
	For every ordered tuple $A$  of elements of $\{1, \dots, n\}$, finite cover $f : \wt U \to U$ and $q \in f^{-1}(H_A)$ the following hold:
	\begin{enumerate}
		\item $\sigma^* \psi_A = \psi_A$;
		\item $\tau^* \psi_A^{\wt U} = \psi_A^{\wt U}$;
		\item $\tau^* \omega_{W, A}^{\wt U} = \omega_{\tau^{-1} W, \ A}^{\wt U}$; and
		\item $\tau^* \ol\omega_{W, A}^{\wt U} = \ol\omega_{\tau^{-1} W, \ A}^{\wt U}$;		 
	\end{enumerate}
	for every $\sigma \in \G$ and $\tau \in \wt \G$.
\end{corollary}
\begin{proof}
	Follows immediately from  \Cref{lem:psi.action,lem:omega.action} and \Cref{eq:indices}.
\end{proof}

\begin{lemma}\label{lem:omega.WA.action}
	Let $A$ be an ordered tuple of elements of $\{1, \dots, n\}$ and let $W \in \pi_0(H_A)$. Then for  all $\sigma \in \G$ we have
	\begin{enumerate}
		\item $\sigma^* \omega_{W, A} = \omega_{\sigma^{-1} W,\, A}$; and
		\item $\sigma^* \ol\omega_{W, A} = \ol\omega_{\sigma^{-1} W,\, A}$.
	\end{enumerate}
\end{lemma}
\begin{proof}
    Recall from \Cref{def:uniqueomega} that 
    $\omega_{W, A}$ is the unique  differential form on $M(\B)$ that
    satisfies
	\[
		f^*(\omega_{W, A}) = \frac{1}{\delta_{A, p}} \sum_{q \in f^{-1}(p)} \omega_{A, q}^{\wt U}	
	\]
	for any (hence all) $p \in W$, where $f$ is any separating cover $f : \wt U \to U$ (\Cref{rem:separating.covers}) and we put $$\delta_{A, p} := \left\vert H_{A, q_0}^{\wt U} \cap f^{-1}(p)\right\vert \textrm{ for some }q_0 \in f^{-1}(p)$$
    (recall \Cref{not:HU} for the definition of $H_{A, q_0}^{\wt U}$).

\paragraph{\em Claim 1. The number $\delta_{A,p}$ does not depend on the choice of $q_0$ in $f^{-1}(p)$}\strut\\
{\em Proof.} Any $q_0'\in f^{-1}(p)$ is the image of $q_0$ by some deck transformation, say $\gamma$, of the covering $f$. Now $\gamma f^{-1}(p) = f^{-1}(p)$ and $\gamma H_{i,q_0}^{\tilde{U}}=  H_{i,\gamma q_0}^{\tilde{U}}$ since $\gamma$ is an automorphism of the lift of $\B$ to $\wt U$. Thus $H_{A, \gamma q_0}^{\wt U} \cap f^{-1}(p)=\gamma (H_{A, q_0}^{\wt U} \cap f^{-1}(p))$, and bijectivity of $\gamma$ ensures that   $H_{A, q_0}^{\wt U} \cap f^{-1}(p)$ and $H_{A, \gamma q_0}^{\wt U} \cap f^{-1}(p)$ are equicardinal. \hfill$\triangle$

\medskip

\paragraph{\em Claim 2: $
 \delta_{A,\sigma p}= \delta_{A,p}$ for every $\sigma\in \G$}\strut\newline
 {\em Proof.} Fix $\tilde{\sigma}\in f^{-1}(\sigma)$. We compute 
$$
 \tilde{\sigma}\left(H_{A, q_0}^{\wt U} \cap f^{-1}(p)\right)=
 H_{A, \tilde{\sigma}q_0}^{\wt U} \cap \tilde{\sigma}f^{-1}(p)
 = 
 H_{A, \tilde{\sigma}q_0}^{\wt U} \cap f^{-1}(\sigma p),
 $$
 where the last equality holds because,  by \Cref{eq:cover.homomorphism},
$f^{-1}(\sigma p) = \wt\sigma f^{-1}(p)$.
 Now the claim follows by bijectivity of $\wt\sigma$ and the independence of the choice of the representative in $f^{-1}(p)$, resp $f^{-1}(\sigma p)$ in the definition of
 $
 \delta_{A,\sigma p}$ and $ \delta_{A,p}
 $ (see Claim 1). \hfill$\triangle$

 Now let us turn to the proof of (i). Fix $\sigma\in \G $ and pick any $\tilde \sigma \in f^{-1}(\sigma)$. Then 
	\begin{align*}
		f^*(\sigma^* \omega_{W, A}) = \tilde \sigma^* f^*(\omega_{W, A}) &= \frac{1}{\delta_{A, p}} \sum\nolimits_{q \in f^{-1}(p)} \tilde \sigma^* \omega_{A, q}^{\wt U} \\
			&\overset{(i)}{=} \frac{1}{\delta_{A, p}} \sum\nolimits_{q \in f^{-1}(p)} \omega_{A, \wt\sigma^{-1} q}^{\wt U} \\
			&\overset{(ii)}{=} \frac{1}{\delta_{A, \sigma^{-1}p}} \sum\nolimits_{q \in f^{-1}(\sigma^{-1}  p)} \omega_{A, q}^{\wt U} = f^*(\omega_{\sigma^{-1}  W, A})
	\end{align*}		
	where equality (i) is
\Cref{lem:omega.action} and (ii) holds by Claim 2 after re-indexing the sum (since $\wt\sigma^{-1}f^{-1}(p)=f^{-1}(\sigma^{-1} p)$).

Now as the differential forms $\sigma^* \omega_{W, A}$ and $\omega_{\sigma^{-1}  W, A}$ do not depend on the choice of point $p \in W$ and $\sigma^{-1} p \in \sigma^{-1}  W$, injectivity of $f^*$ implies
	\[
		\sigma^* \omega_{W, A} = \omega_{\sigma^{-1}  W, A}
	\]
	
	A proof of the claim for $\ol\omega_{W, A}$ can be obtained along the same lines. 
\end{proof}

\begin{proposition}\label{prop:eta.WAB.action}
	Let $\B = \{H_1, \dots, H_n\}$ be an essential toric arrangement in $U$, and let $\G \leq U$ be a finite subgroup such that $\sigma H_i = H_i$ for all $ \sigma \in \G$ and all $1 \leq i \leq n$. Then 
    for all layers $W \in \pi_0(H_A)$ and for all $W$-adapted pairs $(A,B)$ we have
	\[
		\sigma^* \ol\eta_{W, A; B} = \ol\eta_{\sigma^{-1} W, \ A; B}.
	\]
\end{proposition}
\begin{proof}
	Recall that, by definition, $\ol\eta_{W, A; B} = (-1)^{\ell(A, B)} \ol\omega_{W, A} \psi_B$. Now the claim follows applying \Cref{lem:omega.WA.action} and \Cref{cor:psi.omega.A.action} in order to describe the action of $\G$ on $\ol\omega_{W, A}$, resp.~on $\psi_B$.
\end{proof}

\subsection{Actions associated to primitive $p$-covers.}
We are now ready to inspect the action of $\G$ induced on the generators of the cohomology algebra $H^\bul(M(\B), \Z)$ in the special case where the action leaves each hypersurface $H_i\in \B$ invariant but does permute their connected components (this is the situation of interest for the application we have in mind, see \Cref{motivation-setup}). 

\begin{proposition}\label{prop:action.on.cohomology}
Let $\B = \{H_1, \dots, H_n\}$ be an essential toric arrangement in $U \cong (\C^\times)^d$, let $\G \leq U$ be a finite subgroup such that for all $ \sigma \in \G$ we have $\sigma \cdot H_i = H_i$ for $1 \leq i \leq n$. Suppose in addition that the action of $\G$ on the layers of $\B$ of dimension $0$ has at least one nonsingleton orbit. 
Then the following hold:
\begin{enumerate}
	\item $\G$ fixes elementwise the subalgebra of $H^\bul(M(\B), \Z)$ generated by its elements of degree $1$; and
	\item 
 $\G$ acts nontrivially on $H^d(M(\B), \Z)$. %
\end{enumerate}
	In particular the cohomology algebra $H^\bul(M(\B), \Z)$ is not generated in degree one.
\end{proposition}

\begin{remark}[No-broken-circuit sets]\label{rem:nbc} 
In view of the following proof let us recall that, using the setup of \Cref{prop:action.on.cohomology}, the set ${\rm nbc}(\B)$ of no-broken-circuit sets is given by all index sets $A\subset [n]$ such that $A$ is independent and if $A\cup \{i\}$ is dependent for some $i\in [i]$, then  $i\geq \min A$. This is a standard notion in the theory of matroids and hyperplane arrangements, which generalizes to the setup of toric arrangements, see, e.g., \cite{CDDMP20} and references therein.
\end{remark}

\begin{proof}
By \Cref {prop:generators}, $H^\bul(M(\B), \Z)$ is generated by the cohomology classes:
\[
	e_{W, A ; B} = [\ol\eta_{W, A; B}]
\]
where $W$ is a layer of $\B$ and $(A,B)$ is a $W$-adapted pair. 
The generator $e_{W, A; B}$ has degree $|A\sqcup B|$, so the only generators of degree $1$ must either have  $A=\emptyset$ (so $W$ is equal to the full torus $U$ and $B$ must be a singleton) or $A = \{i\}$ for some $i\in [n]$ (so that $W = H_i$ and $B$ is empty). In other words the generators in degree $1$ are the elements 
$$
e_{U, \emptyset;\{i\}} \quad
\textrm{ and }
\quad
	e_{H_i, \{i\};\emptyset}\quad
 \textrm{ for }
 i=1,\ldots n.
$$

Since for every $\sigma \in \G$ we have $\sigma \cdot H_i = H_i$ for all $H_i \in \B$ and obviously $\sigma \cdot U = U$, by Proposition \ref{prop:eta.WAB.action} we obtain:
\begin{align*}
	\sigma \cdot e_{U,\emptyset;\{i\}} &= [\sigma^* \ol\eta_{U,\emptyset;\{ i\}}] = [\ol\eta_{U,\emptyset; \{i\}}] = e_{U,\emptyset;\{i\}} \\
	\sigma \cdot e_{H_i,\{i\}; \emptyset} &= [\sigma^* \ol\eta_{H_i,\{i\}; \emptyset}] = [\ol\eta_{H_i,\{i\}; \emptyset}] = e_{H_i,\{i\}; \emptyset}.
\end{align*}
Hence the subalgebra generated by these elements is fixed pointwise by the action of $\G$, and (i) follows.\medskip \\
Now consider the set $\lay_0 =\{p\in \lay(\B) \mid \dim(p)=0\}$ of all layers of $\B$ of dimension $0$.
By assumption there are at least two distinct elements of $\lay_0$ in the same $\G$-orbit, say $p'$ and $p''=\sigma p' \neq p'$ for some $\sigma\in \G$. Then, for every $H\in \B$ with $p'\in H$ we have $p''\in \sigma H$, hence the pair $(A,\emptyset)$ is adapted for $p'$ if and only if it is adapted for $p''$. Let $A\in \rm{nbc}(\B)$ be such that $(A,\emptyset)$ is $p'$-adapted (such an $A$ exists, see e.g.~\cite[Theorem 6.14]{CDDMP20}). In order to prove (ii) it is enough to show that
\begin{equation}\label{permeta}
\sigma^*\ol\eta_{p'', A; \emptyset}=\ol\eta_{p', A; \emptyset}\neq \ol\eta_{p'', A; \emptyset}\end{equation} 
The equality on the left-hand side holds by \Cref{prop:eta.WAB.action}. In order to prove the inequality in \Cref{permeta} we show that the  cohomology classes $e_{p, A; \emptyset}$ with $p\in \lay_0$ and $A\in {\rm nbc}(\B)$ are distinct. 
To this end, following \cite{CDDMP20} we consider the filtration 
\[
	{\mathcal F}_i H^\bul(M(\B), \Z) := \sum_{j \leq i} H^\bul(U, \Z) \cdot H^j(M(\B), \Z) 
\] 

whose associated graded modules are given, for each degree $k$, by
\[
	{\rm gr}_k(H^\bul(M(\B), \Z)) = \bigoplus_{\underset{{\rm codim}(W) = k}{W \in \lay(\B)}} H^\bul(W) \otimes H^k(M(\B[W]))
\]
where $\B[W]$ is the linear hyperplane arrangement induced on any tangent space $T_xU$ for $x\in W$ generic. Then, specializing to $k=d=\dim(U)$,
\[
	{\rm gr}_d(H^\bul(M(\B), \Z))=\bigoplus_{p\in \lay_0
 } H^\bul(p) \otimes H^d(M(\B[p])) \simeq 
 \bigoplus_{p\in \lay_0
  }  H^d(M(\B[p]))
\]
(all $H^\bul(p)$ are trivial since $\dim(p)=0$). Here
 $H^d(M(\B[p])$ is the top-degree part of the Orlik-Solomon algebra of the central arrangement $\B[p]$, and  \cite[Lemma 5.14]{CDDMP20} shows that $e_{p, A; \emptyset}$ maps to (a scalar multiple of) the standard Orlik-Solomon generator associated to the independent set $A$. In particular, if ${\rm nbc} (\B[p])$ denotes the set of no-broken circuit sets (called ``$\chi$-independent'' in \cite{OT92}) of the hyperplane arrangement $\B[p]$, then  ${\rm nbc} (\B[p]) \subseteq {\rm nbc} (\B)$: indeed $(A,\emptyset)$ is $W$-adapted for $A\in \mathrm{nbc} (\B)$ if and only if $A\in\mathrm{nbc}(\B[p])$. Now by \cite[Theorem 3.43 and Theorem 3.119]{OT92} the classes 
 $$e_{p, A; \emptyset},
 \textrm{ with }
 p\in \lay_0, 
 \textrm{ and }
A\in {\rm nbc}(\B[p])
 $$
 form a full-rank independent subset of ${\rm gr}_d(H^\bul(M(\B), \Z))$. In particular, $e_{p', A; \emptyset}$ and $e_{p'', A; \emptyset}$ are distinct. This concludes the proof of \Cref{permeta} and thus of claim (ii).

\end{proof}

Putting together the results of the previous sections we obtain the following theorem.

\begin{theorem}\label{th:p.cover.cohomology.not.one.generated}
	Let $\B$ be a central, essential, primitive toric arrangement in $T \cong (\C^\times)^d$. Then for prime numbers $p$ for which the map $\phi_p : \B \to \mathbb P^1(\F_p^d)$ defined in \Cref{prop:existence.primitive.p.cover} is not surjective, there exists a primitive $p$-cover $f : U \to T$ such that the cohomology algebra $H^\bul(M(\B^U), \Z)$ of the lifted arrangement $\B^U$ is not generated in degree one. 
\end{theorem}
\begin{proof}

From the proof of \Cref{prop:existence.primitive.p.cover} we know that we can find a primitive $p$-cover $f : U \to T$ of degree $p$, i.e. a cyclic cover. Consider the cyclic deck transformation group $\G$, which can be identified with the fiber $f^{-1}({\bf 1}) \leq U$ of the identity element ${\bf 1} \in T$, acting by translations on the torus $U$. Notice that the action is fixed-point-free: in particular, since ${\bf 1}$ is a layer of $\B$ of dimension $0$, the set $f^{-1}({\bf 1})$ is a $\G$-orbit of layers of $\B^U$ of dimension $0$ with $p>1$ elements. 
Moreover, for every $H \in \B$ the lift $\hat H = f^{-1}(H)$ is connected because the cover is primitive.
Since the action of $\G$ permutes the components of each lifted hypertorus, we have that $\G \cdot \hat H = \hat H$ for every $\hat H\in \B^U$. We are then in the hypotheses of \Cref{prop:action.on.cohomology}, hence we conclude that $H^\bul(\B^U, \Z)$ is not generated in degree one.
\end{proof}

\section{The main theorem and applications to pure braid groups
}\label{sec:bloch.kato}

Here we apply the results of Section 4 the class of supersolvable toric arrangements. Recall that the complement $M(\B)$ of a supersolvable toric arrangement $\B$ is a $K(\pi, 1)$ space, thus any consideration about the cohomology of $M(\B)$ translates to the cohomology of the fundamental group $G(\B)$. We will make use of this observation in order to obtain statements about the Bloch-Kato property of pro-$p$ completions of $G(\B)$, first in general and then in the special case when $G(\B)$ is a pure braid group. \medskip\\
In particular, the techniques developed in the previous sections allow us to construct for almost all primes $p$ finite covers of index $p$ of supersolvable arrangements with cohomology not generated in degree $1$ (hence non quadratic). Passing to these finite covers will provide finite index subgroups with non-quadratic cohomology witnessing the failure of the Bloch-Kato property.

\subsection{On the Bloch-Kato property of supersolvable arrangements}

We start by relating the $\F_p$ cohomology of a supersolvable arrangement's complement with that of the pro-$p$ completion of its fundamental group.

\begin{proposition}\label{prop:cohom.comparison}
    Let $\B$ be a supersolvable toric arrangement in $T\cong (\C^*)^d$  
    and let $G(\B)$ be the fundamental group of the complement $M(\B)$. Then for any prime number $p$ we have an isomorphism of graded $\F_p$-algebras:
        \[
            H^\bul(G(\B)_{\hat p}, \F_p) \cong H^\bul(M(\B), \F_p)
        \]
\end{proposition}
\begin{proof}
    Via \Cref{prop:polyfree.and.Kpi1}, supersolvability of $\B$ implies that the fundamental group $G(\B)$ is an iterated semidirect product of free groups. Such groups are $p$-good in the sense of Serre for all primes $p$. In particular, the pro-$p$ completion homomorphism $\rho_{G(\B), p} : G(\B) \to G(\B)_{\hat p}$ described in \S\ref{ssec:procom} induces an isomorphism of graded $\F_p$-algebras
    \[
        (\rho_{G(\B),p})^* : H^\bul(G(\B)_{\hat p}, \F_p) \cong H^\bul(G(\B), \F_p).
    \]
    Again from \Cref{prop:polyfree.and.Kpi1} we know that the complement manifold $M(\B)$ of a supersolvable toric arrangement is a $K(G(\B), 1)$ space, thus we have an isomorphism:
    \[
        H^\bul(G(\B), \Z) \cong H^\bul(M(\B), \Z)
    \]
    Finally, since the integral cohomology algebra is torsion-free, we obtain the assertion by tensoring the cohomology with $\F_p$.
\end{proof}

\begin{corollary}\label{cor:quadratic.cohom}
If $\B$ is a strictly supersolvable toric arrangement, the pro-$p$ group $G(\B)_{\hat p}$ has quadratic $\F_p$-cohomology.
\end{corollary}
\begin{proof}
By \Cref{prop:strictly.supersolvable.quadratic}, the cohomology algebra of $M(\B)$ over $\Q$ is quadratic (and such cohomology is torsion-free hence it is also quadratic over $\Z$ and $\F_p$). The claim now follows by \Cref{prop:cohom.comparison}
\end{proof}
\noindent
In light of \Cref{cor:quadratic.cohom} it is natural to ask whether any of such pro-$p$ groups have the Bloch-Kato property. The following theorem states that, generically, this property does not hold.

\begin{theorem}\label{th:arrangement.group.non.bk} Let $\B$ be an essential, primitive, supersolvable toric arrangement in $T = (\C^\times)^d$ with defining set of characters $X$
and recall the map $\phi_p : X \to \mathbb P^1(\F_p^d)$ from \Cref{prop:existence.primitive.p.cover}. Let $G(\B)$ be the fundamental group of the complement space $M(\B)$. Then for all prime numbers $p$ for which $\phi_p$ is not surjective, the pro-$p$ completion $G(\B)_{\hat p}$ is not Bloch-Kato.
\end{theorem}
\begin{proof}    

    For any essential and primitive toric arrangement $\B$ we know from \Cref{th:p.cover.cohomology.not.one.generated} that for all primes $p$ for which the map $\phi_p : X \to \mathbb P^1(\F_p^d)$ is not surjective, there exists a primitive $p$-cover $f : U \to T$. Moreover, if $\B^U=f^{-1}(\B)$ is the arrangement in $U$ obtained by lifting $\B$, we have that the fundamental group $G(\B^U)$ can be identified with a subgroup $K$ of $G(\B)$ of index $p$ and that the cohomology algebra $H^\bul(M(\B^U), \Z)$ is not generated in degree $1$ -- in particular, it is not quadratic.

    Now, assume $\B$ to be supersolvable. Then  by \Cref{BUss} the arrangement $\B^U$ is supersolvable as well, and we know from  \Cref{prop:cohom.comparison} that $H^\bul(K_{\hat p}, \F_p)$ is not quadratic. For all relevant primes $p$, consider  the pro-$p$ completion homomorphism $\rho_{G(\B), p} : G(\B) \to G(\B)_{\hat p}$. Using the fact that $G(\B)$ is residually $p$ and that $K$ has finite index in $G(\B)$, the image $\rho_{G(\B), p}(K)$ is isomorphic with the pro-$p$ completion $K_{\hat p}$  %
    \cite[Lemma 3.1.4 (a) and Proposition 3.2.2 (a)]{RZ10}. In conclusion, we have found a closed subgroup $K_{\hat p}$ of the pro-$p$ completion $G(\B)_{\hat p}$ with non-quadratic $\F_p$ cohomology. In particular $G(\B)_{\hat p}$ is not Bloch-Kato. 
\end{proof}

\subsection{Application: pure braid groups}

Consider, as in \Cref{ex2}, the arrangements $\tbraid{n}$ of type $A_n$ in $T \cong (\C^\times)^n$, for $n > 2$. The essentialized arrangement $\tbraid{n}^e$ in the torus $(\C^\times)^{(n-1)}$  consists of $\binom{n}{2} = \frac{n(n-1)}{2}$ hyperplanes. 

\begin{remark}[The cases $n>3$ or $p>2$]\label{rem:largenp}
Whenever $n > 3$ and $p$ is a prime number, or $n = 3$ and $p > 2$, we have: 
\[
    |\mathbb P^1(\F_p^{n-1})| = \frac{p^{n-1}-1}{p-1} > \frac{n(n-1)}{2}
\]
thus the map $\phi_p$ can never be surjective and by \Cref{th:arrangement.group.non.bk} the fundamental group $G(\B_n^e)_{\hat p}$ is not Bloch-Kato. 
\end{remark}

In order to treat the remaining case ($n=3$, $p=2$) %
we  exhibit a covering $U$ of $(\C^\times)^2$ of degree $4$ such that the complement of the lifted arrangement $(\tbraid{3}^e)^U$ has non quadratic $\mathbb F_2$-cohomology.

\begin{lemma}\label{lem:explicit.cover}
    Let $U$ be the covering space of $T \cong (\C^\times)^2$ defined as in \S\ref{sec:lifting} by the injection $\mathbb Z^2 \to \mathbb Z^2$, $z\mapsto Az$ where $$A := \left(\begin{matrix}
        2 & 1 \\ 0 & -4
    \end{matrix}\right).$$
    Then the cohomology algebra of the complement of the arrangement $(\tbraid{3}^e)^U$ lifted from $\tbraid{3}^e$ is not generated in degree $1$.  
\end{lemma}
\newcommand{\nascosto}[1]{}
\begin{proof}    
\def\nonso{\widetilde \B}
Write $\nonso$ for $(\tbraid{3}^e)^U$. This arrangement has been investigated in \Cref{exlift} and \Cref{fig3}. 
\nascosto{
    A quick computation based on \Cref{ex2} and \Cref{lem:essentialization} shows that the set of defining characters of $\tbraid{3}^e$ can be chosen as $X = \{\chi_1 = (1, 0), \chi_2 = (0, 1), \chi_3 = (1, 1)\}$, corresponding to the hypersurfaces:
    \begin{align*}
        H_1 &:= \{ (t_1, t_2) \in T \mid t_1 = 1 \} \\
        H_2 &:= \{ (t_1, t_2) \in T \mid t_2 = 1\} \\
        H_3 &:= \{ (t_1, t_2) \in T \mid t_1t_2 = 1 \}.
    \end{align*}

    The lifts via the covering map corresponding to the 
    matrix $A$ are then given by
    \begin{align*}
        H_1^U &:= \{ (u_1, u_2) \in U \mid u_1^2 = 1 \} \\
        H_2^U &:= \{ (u_1, u_2) \in U \mid u_1u_2^{- 4} = 1\} \\
        H_3^U &:= \{ (u_1, u_2) \in U \mid u_1^3u_2^{-4} = 1 \}.
    \end{align*}
    \def\nonso{\widetilde \B}

    In the remainder of this proof we will write $\nonso$ for $(\tbraid{3}^e)^U$. Notice that while $\nonso$ is supersolvable like $\tbraid{3}^e$, it is no longer strictly supersolvable. Indeed since $M$-ideals must be join-closed, any $M$ ideal must miss at least one among $H^U_2$ and $H^U_3$. But the intersection of either of those with any other hypersurface is disconnected.

\begin{figure}[h] %
\centering
\begin{tikzpicture}[x=15em,y=15em]

    \node[anchor=center] (A2) at (0, 1) {};
    \node[anchor=center] (B2) at (0.5, 1) {};
    \node[anchor=center] (C2) at (0.5, 1.5) {};
    \node[anchor=center] (D2) at (0, 1.5) {};

\fill[gray!10] (A2.center) -- (B2.center) -- (C2.center) -- (D2.center) -- (A2.center);
\begin{scope}
    \draw[very thick] (A2.center) -- (D2.center);
    \draw[very thick] (A2.center)++(0.25,0) -- (0.25, 1.5);
    \clip (A2.center) -- (B2.center) -- (C2.center) -- (D2.center) -- (A2.center);
    \draw[thick] (A2.center) --+ (2, 0.5);
    \draw[thick] (A2.center)++(0,0.25) --+ (2, 0.5);
    \draw[thick] (A2.center)++(0,0.125) --+ (2, 0.5);
    \draw[thick] (A2.center)++(0,0.375) --+ (2, 0.5);    
    
    \draw[] (A2.center) --+ (2, 1.5);
    \draw[] (A2.center)++(-0.5, 0) --+ (2, 1.5);    
    \draw[] (A2.center)++(-0.333, 0) --+ (2, 1.5);    
    \draw[] (A2.center)++(-0.166, 0) --+ (2, 1.5);
    \draw[] (A2.center)++(0.333, 0) --+ (2, 1.5);
    \draw[] (A2.center)++(0.166, 0) --+ (2, 1.5);
\end{scope}

    \node[anchor=center] (A1) at (0, 0) {};
    \node[anchor=center] (B1) at (0.5, 0) {};
    \node[anchor=center] (C1) at (0.5, 0.5) {};
    \node[anchor=center] (D1) at (0, 0.5) {};
    
\fill[gray!10] (A1.center) -- (B1.center) -- (C1.center) -- (D1.center) -- (A1.center);

    \draw[very thick] (D1.center) -- (A1.center);
    \draw[] (C1.center) -- (A1.center);
    \draw[thick] (A1.center) -- (B1.center);
\node[rotate=-90, scale=2] (arrow) at (0.25, 0.75) {$\twoheadrightarrow$};
\end{tikzpicture}
\quad
\begin{tikzpicture}[x=5em,y=4.5em]
\node[anchor=center] (O) at (0,-1) {$U$};
\node[anchor=center] (H2) at (-1.5,0) {$H_2^U$};
\node[anchor=center] (H1x) at (-0.5,0) {$H_{1, x}^U$};
\node[anchor=center] (H1y) at (0.5,0) {$H_{1, y}^U$};
\node[anchor=center] (H3) at (1.5,0) {$H_3^U$};
\node[anchor=center] (P) at (-2,1.25) {$P$};
\node[anchor=center] (Q) at (-1.5,1.25) {$Q$};
\node[anchor=center] (R) at (-1,1.25) {$R$};
\node[anchor=center] (S) at (-.5,1.25) {$S$};
\node[anchor=center] (P') at (.5,1.25) {$P'$};
\node[anchor=center] (Q') at (1,1.25) {$Q'$};
\node[anchor=center] (R') at (1.5,1.25) {$R'$};
\node[anchor=center] (S') at (2,1.25) {$S'$};

\draw (H2.north) -- (P.south);
\draw (H2.north) -- (Q.south);
\draw (H2.north) -- (R.south);
\draw (H2.north) -- (S.south);
\draw (H2.north) -- (P'.south);
\draw (H2.north) -- (Q'.south);
\draw (H2.north) -- (R'.south);
\draw (H2.north) -- (S'.south);

\draw (H3.north) -- (P.south);
\draw (H3.north) -- (Q.south);
\draw (H3.north) -- (R.south);
\draw (H3.north) -- (S.south);
\draw (H3.north) -- (P'.south);
\draw (H3.north) -- (Q'.south);
\draw (H3.north) -- (R'.south);
\draw (H3.north) -- (S'.south);

\draw (H1x.north) -- (P.south);
\draw (H1x.north) -- (Q.south);
\draw (H1x.north) -- (R.south);
\draw (H1x.north) -- (S.south);

\draw (H1y.north) -- (P'.south);
\draw (H1y.north) -- (Q'.south);
\draw (H1y.north) -- (R'.south);
\draw (H1y.north) -- (S'.south);

\draw (O.north) -- (H1x.south);
\draw (O.north) -- (H1y.south);
\draw (O.north) -- (H2.south);
\draw (O.north) -- (H3.south);

\node (bottom) at (0,-2) {};
\end{tikzpicture}

\caption{The M\"obius function has value $1$ at the bottom element and $-1$, resp. $2$ at each element of rank $1$, resp.~$2$. Thus the characteristic polynomial of the poset is $q(t)=t^2 - 4t + 16$ and so the Poincaré polynomial of the complement is 
$t^2q(-\frac{t+1}{t})=21t^2 + 6t + 1$ (see, e.g., \cite[Corollary 5.12]{Moc12})  
}\label{fig3}
\end{figure}
}
In particular, the Poincar\'e polynomial for the cohomology of $M(\nonso)$ can be computed using the characteristic polynomial of  the poset $\lay({\nonso})$ (see caption of \Cref{fig3}) obtaining:
    \[
        P_{\nonso}(t) = 1 + 6 t + 21 t^2.
    \]
    \noindent
    If $H^\bul(M(\nonso), \Z)$ were to be generated in degree $1$, the cup product $\cup : H^1(M(\nonso), \Z) \wedge H^1(M(\nonso), \Z) \to H^2(M(\nonso), \Z)$ were to be surjective, implying that the rank of $H^2(M(\nonso), \Z)$ should be at most $\binom{\rk H^1(M(\nonso), \Z)}{2}$. However from the Poincar\'e polynomial one can read that the ranks of $H^1(M(\nonso), \Z)$ and $H^2(M(\nonso), \Z)$ are respectively $6$ and $21 > 15 = {6\choose 2}$, so $H^\bul(M(\nonso), \Z)$ cannot possibly be generated only by elements of degree $1$.
\end{proof}

\begin{corollary}\label{cor:n3p2}
    The pro-$2$ completion of the fundamental group $G(\B_3^e)$ is not Bloch-Kato.
\end{corollary}
\begin{proof}
    Recall that, since $G(\tbraid{3}^e)$ is an iterated semidirect product of free groups, it is residually-$2$. Moreover the subgroup $H \cong G((\tbraid{3}^e)^U)$ has index $|\det(A)| = 8$ in $G(\tbraid{3}^e)$, so its image in the pro-$2$ completion $G(\tbraid{3}^e)_{\hat 2}$ is $\rho_{2, G(\tbraid{3}^e)}(H) = H_{\hat 2}$ (see \cite[Lemma 3.1.5 (a)]{RZ10}). 
    By \Cref{prop:cohom.comparison}, the cohomology of $H_{\hat 2}$ with $\F_2$ coefficients is isomorphic to the cohomology of $H \cong G((\tbraid{3}^e)^U)$, which is not $1$-generated by \Cref{lem:explicit.cover}. We conclude that $G(\tbraid{3}^e)_{\hat 2}$ is not Bloch-Kato.
\end{proof}
\noindent
Putting together the results of this section, we have the following.

\begin{theorem}\label{th:braid.bk}
    For all prime numbers $p$ and all $n > 2$:
    \begin{enumerate}
    \item the pro-$p$ completion of the pure braid group $(\pbraid{n+1})_{\hat p}$ is not Bloch-Kato;
    \item the pro-$p$ completion of the pure mapping class group $(\mcg{n+2})_{\hat p}$ is not Bloch-Kato.
    \end{enumerate}
    Moreover for all primes $p$, the pro-$p$ completion of the pure braid groups $\pbraid{k}$ for $1\leq k \leq 3$ as well as of $\mcg{k}$ for $1\leq k\leq 4$ is Bloch-Kato.
\end{theorem}
\begin{proof}
By \Cref{rem:largenp} and \Cref{cor:n3p2} we know that
for all $n > 2$ and for all primes $p$, the pro-$p$ completion of the fundamental group of the essentialized arrangement $\B_n^e$ is not Bloch-Kato. Moreover, from \Cref{lem:essentialization}.(3) we know that the fundamental group of $\B_n$, which equals $\pbraid{n+1}$ by \Cref{ex3}, can be expressed as $G(\B_n^e) \times \Z$. When we take the pro-$p$ completion we obtain 
\[
    (\pbraid{n+1})_{\hat p} \cong G(\B_n^e)_{\hat p} \times \Z_p
\] 
where $\Z_p$ is the pro-$p$ completion of $\Z$, that is the group of $p$-adic integers. 

For claim (i) notice that, since $G(\B_n^e)_{\hat p}$ is not Bloch-Kato, there exists a closed subgroup $K \leq G(\B_n^e)_{\hat p}$ with non-quadratic $\F_p$-cohomology algebra. Such a subgroup $K \leq G(\B_n^e)_{\hat p} \times \Z_p$ 
witnesses the failure of the Bloch-Kato property of $(\pbraid{n+1})_{\hat p}$.
Claim (ii) follows immediately from the isomorphism  $\pi_1(M(\ess{\B_n}))\simeq \mcg{n+2}$ discussed in \Cref{rem:mcg}. 

Finally we study the remaining cases. The groups $\pbraid{1}$ and $\mcg{k}$ for $k=1,2,3$ are trivial %
\cite[p.~101, Section 4.2.4]{Farb}
and so there is nothing to show. 
Moreover, $\pbraid{2} \cong \Z$ and so its pro-$p$ completion is Bloch-Kato for all primes $p$. The Fadell-Neuwirth split exact sequence $1\to F_2 \to \pbraid{3}\to \pbraid{2} \to 1$ implies that $\pbraid{3} \cong F_2 \times \Z$, whose pro-$p$ completion is Bloch-Kato for all primes $p$ (see \cite[Definition 4.9]{MPQN21} and references therein). Finally, again by \cite[p.~101, Section 4.2.4]{Farb} we have  $\mcg{4}_{\hat p} \cong F_2$, and the claim follows.  

\end{proof}

\bibliographystyle{plain}
\bibliography{refs}

\end{document}